\documentclass[a4paper,11pt]{article}
\usepackage[a4 paper, margin=2.5cm]{geometry}

\usepackage[normalem]{ulem}

\usepackage{hyperref}

\usepackage{amsmath,amsfonts,amsthm,amssymb}
\usepackage[normalem]{ulem}
 \usepackage{soul}
\usepackage{cite}
\usepackage{color,enumitem,graphicx}

\usepackage[english]{babel}
\usepackage[dvipsnames]{xcolor}

\newcommand\myshade{85}
\colorlet{mylinkcolor}{violet}
\colorlet{mycitecolor}{YellowOrange}
\colorlet{myurlcolor}{Aquamarine}

\hypersetup{
	linkcolor  = mylinkcolor!\myshade!black,
	citecolor  = mycitecolor!\myshade!black,
	urlcolor   = myurlcolor!\myshade!black,
	colorlinks = true,
}

\usepackage[T1]{fontenc}
\usepackage{lmodern}

\usepackage{fixmath}

\numberwithin{equation}{section}

\newtheorem{theorem}{Theorem}[section]
\newtheorem{lemma}[theorem]{Lemma}
\newtheorem{proposition}[theorem]{Proposition}

\theoremstyle{definition}
\newtheorem{definition}[theorem]{Definition}

\newtheorem{notat}[theorem]{Notation}

\theoremstyle{remark}
\newtheorem{remark}[theorem]{Remark}

\newcommand{\R}{\mathbb{R}}
\newcommand{\C}{\mathbb{C}}

\newcommand{\cD}{{\mathcal D}}

\newcommand{\N}{\mathbb N}

\DeclareMathOperator{\esssup}{ess\, sup}

\DeclareOldFontCommand{\bf}{\normalfont\bfseries}{\mathbf}
\DeclareMathOperator{\codim}{codim}

\numberwithin{equation}{section}

\DeclareMathOperator*{\supp}{supp}

\DeclareSymbolFont{rsfs}{U}{rsfs}{m}{n}
\DeclareSymbolFontAlphabet{\mathscr}{rsfs}

\def\D{{\nabla}}
\def\e{{\varepsilon}}
   
\newcommand{\beq}{\begin{equation}}
\newcommand{\eeq}{\end{equation}}

\DeclareMathOperator*{\const}{const}
\DeclareMathOperator*{\spann}{span}

\begin{document}

\title{Normalized Schr\"odinger equations with mass-supercritical nonlinearity in exterior domains}

\author{Luigi Appolloni
\\
{\small \textit{University of Leeds, School of Mathematics}}
\\
{\small \textit{Woodhouse, Leeds LS2 9JT, United Kingdom}}
\\
{\small\texttt{l.appolloni@leeds.ac.uk}}
\\ \\
Riccardo Molle
\\
{\small \textit{Dipartimento di Matematica, Universit\`a di Roma ``Tor Vergata''}}
\\
{\small\textit{Via della Ricerca Scientifica n. 1, 00133 Roma, Italy}}
\\
{\small \texttt{molle@mat.uniroma2.it}}
}
 
 
\maketitle


\begin{abstract}
We consider the problem $-\Delta u+\lambda u=u^{p-1}$, where $u\in H^1_0(\Omega)$ verifies $\|u\|_{L^2}=m>0$, and $\lambda\in [0,+\infty)$. 
Here, $\R^N\setminus\Omega$ is nonempty and compact.
We prove the existence of a solution with a constrained Morse index lower than or equal to $N+1$, both in the case $m$ fixed and $\R^N\setminus\Omega$ in a small ball and in the case $\Omega$ fixed and $m$ large.
\end{abstract}


  \section{Introduction}


In this paper, we investigate the existence of weak solutions for 
\begin{equation} \label{Pm}
\tag{$P_m$}
    \begin{cases}
-\Delta u +\lambda u = |u|^{p-2}u & \text{in} \ \Omega \\
\lambda\in\R,\quad   u\in S_m:=\{u\in H^1_0(\Omega)\ :\  
\Vert u \Vert_{L^2} = m \}
    \end{cases}
\end{equation}
where $\Omega \subset \R^N $ is an exterior domain, i.e. there exists a $\varrho>0$, such that $\R^N \setminus \Omega \subset B_\varrho(0)$,  $N \geq 3$,  $ 2 + \frac{4}{N}<p<\tfrac{2N}{N-2}$, and $m>0$.

\smallskip

It is classical that problem \eqref{Pm} comes out studying the nonlinear homogeneous model of the  Schr\"odinger equation
\beq
\label{SchEq}
i \Phi_t+\Delta\Phi+f(\Phi)=0,\quad \Phi:\R^N\times [0,+\infty)\to \C,\quad \Phi(\cdot ,0)= u_0,
\eeq
in the stationary case, that is looking for the standing waves solutions of the form
\beq
\label{ansatz}
\Phi(x,t)=e^{i\lambda t}u(x),\qquad t\ge 0,\ u:\R^N\to\R.
\eeq
Since for the time-dependent solutions $\Phi(x,t)$ of \eqref{SchEq} the {\em mass} $\int |\Phi(x,t)|^2dx$ is preserved, as $t\in [0,+\infty)$, it is natural to fix in \eqref{ansatz} the quantity $\int u^2dx=m^2$, where $m>0$ is related to the initial condition.
In this framework, the frequency $\lambda$ in  \eqref{ansatz} appears as an unknown in problem \eqref{Pm}, and corresponds to the Lagrange multiplier arising in the variational analysis of the problem under the mass-constraint, namely finding solutions as critical points of the energy functional
$$
    I  (u):=\frac{1}{2} \int_{\R^N}  | \nabla u |^2 \, dx - \frac{1}{p} \int_{\R^N}  | u |^p \, dx 
$$ 
constrained on
\begin{equation*}
    \widetilde{S}_m:= \left\lbrace u \in {H^1(\R^N) \ :\ } \int_{ \R^N} |u|^2 \, dx = m^2 \right\rbrace.
\end{equation*}
We refer to \cite{BeLi83ARMA1,Ca03book,CaLi82CMP,MR4304693, MRV22JDE}  and the references therein for further physical motivations. 

\smallskip

The existence of solutions with prescribed $L^2$ norm, also referred to in the literature as \textit{normalized solutions}, has been extensively investigated in the last few decades in the whole Euclidean space. Following the variational approach proposed above, it is easy to see that the so-called \textit{mass critical} exponent $2+\tfrac{4}{N}$ plays an essential role in characterizing the problem.  
Indeed, in the \textit{mass sub-critical case}, which corresponds to $2<p<2+\tfrac{4}{N}$, the functional is bounded from below by the Gagliardo-Nirenberg inequality, and the problem can be tackled by minimization (see, for instance, the seminal work \cite{MR0778974}).

 In the \textit{mass-supercritical} case $2+\tfrac4N<p<\tfrac{2N}{N-2}$ the functional is unbounded from below, and even if one is able to prove the existence of Palais-Smale sequences, it is not clear in general whether they are bounded. To overcome these difficulties, in the pioneering work \cite{Je97NA}, Jeanjean, studying problem \eqref{Pm} in $\mathbb{R}^N$ with a more general nonlinearity, introduced an approach relying on the Pohozaev identity and scaling arguments to construct particular bounded Palais-Smale sequences. Under suitable assumptions on the nonlinearity, he proved the existence of a mountain-pass type solution, i.e., a solution with constrained Morse index 1.  
Later on, Jeanjean's method was extended to problems with different nonlinearities (see f.i. \cite{MR4150876,MR4096725,MR4107073}, and references therein), on metric graphs \cite{MR4755505,MR3959930}, systems \cite{MR3895385,MR3539467,MR3547674}, non-autonomous equations \cite{LM_PRSE_23,MR4304693,IkMi20CalcVar},
and non-local operators \cite{MR4232664}.

However, in an exterior domain, this approach does not work. In fact, scaling arguments cannot be applied keeping the domain fixed, and the Pohozaev identity contains additional boundary terms that prevent constructing a specific Palais-Smale sequence as done in \cite{Je97NA}. 
Moreover, it turns out that the mountain-pass level coincides with that in $\mathbb{R}^N$, and it is straightforward to see that it cannot be attained. 
As a consequence, it is necessary to look for solutions at higher energy levels. A similar issue arises when studying problem \eqref{Pm} in $\mathbb{R}^N$ with a non-negative vanishing potential. 
To tackle this issue, the second author and alter in \cite{MR4304693} showed that the functional related to the problem exhibits a linking geometry that aligns with a solution of constrained Morse index $N+1$.
This linking structure assures a Palais-Smale sequence, but establishing its compactness is not straightforward.
To carry out the proof, the boundedness of the Palais-Smale sequence is first established, and then it is verified that the Lagrange multiplier $\lambda$ is positive. This allows the compactness analysis in the unconstrained case \cite{BC87ARMA} to be applied. 
A key tool, both for proving the boundedness of the Palais-Smale sequence and for ensuring $\lambda > 0$, was an ``almost Pohozaev'' identity, which, however, cannot be used when $\Omega$ is an exterior domain. 
Motivated by the study of the existence of normalized solutions on metric graphs in the mass-supercritical case, Borthwick, Chang, Jeanjean, and Soave generalized in \cite{MR4748624} the celebrated Struwe's monotonicity trick to the case of constrained functionals. 
As a result, they developed an alternative approach to study solutions with prescribed mass using a new technique to prove the existence of bounded Palais-Smale sequences that do not require the Pohozaev identity.  
This paper investigates the presence of solutions to problem \eqref{Pm}  employing this novel variational method.
To achieve this, for $\eta > 0$ we introduce the family of approximating problems  
\begin{equation}
  \label{eq U} \tag{$P_{m,\eta}$}
\begin{cases}
-\Delta u + \lambda_\eta u = \eta |u|^{p-2} u & \text{in} \ \Omega, \\
\lambda_\eta \in \mathbb{R},\quad u \in S_m, & 
\end{cases}
\end{equation} 
 whose solutions are critical points of the energy functional
\begin{equation}
\label{1727}
    I_\eta (u)=\frac{1}{2} \int  | \nabla u |^2 \, dx - \frac{\eta}{p} \int  | u |^p \, dx 
\end{equation}
constrained on $S_m$.
Inspired by  \cite{MR4304693}, we construct a new family (independent of $\eta$) of subsets of $\mathbb{R}^N$ and prove that the min-max level defined through this family admits bounded Palais-Smale sequences for almost every $\eta $ sufficiently close to $1$ (see \S \ref{S3}). 
Once the existence of bounded Palais-Smale sequences has been established, we still face some compactness issues due to the unboundedness of the domain. To overcome these difficulties, we rely on a splitting lemma, originally proved by Benci and Cerami in \cite{BC87ARMA}, together with additional assumptions either on the size of the complement of the domain or on the mass. 
We emphasize that our approach still requires proving the positivity of the Lagrange multiplier. 
As deriving this from a Pohozaev identity argument, as achieved in \cite{MR4755505} and \cite{MR4601303}, is not feasible, we alternatively prove its non-negativity by utilizing information on the \textit{approximate Morse index}, made accessible by the variational method introduced in \cite{MR4748624}.
We also point out that, compared to the case of metric graphs, the Morse index information available here is not sufficient to guarantee the positivity of the Lagrange multiplier. 
As a consequence, we must restrict the class of domains under consideration by requiring a non-existence condition for a certain class of differential equations (see Remark \ref{Nonexistence} for further details). Finally, to obtain solutions to the problem with $\eta=1$, we consider a sequence $\eta_n \nearrow 1$ and perform an asymptotic analysis of the approximated solutions $u_\eta$, generalizing the approach in \cite{MR2825606} to the case of unbounded domains with a bounded complement. 
 Namely, in Proposition \ref{th3}, it will be proved that a blow-up event is restricted to occurring at only a finite number of points. 
 This limitation arises due to the upper bound on the Morse index of $u_\eta$, and contradicts the assumption of positive mass.

\smallskip

This paper provides a positive answer to a question raised in \cite{MRV22JDE} regarding the existence of positive solutions to problem \eqref{Pm} in exterior domains. 
We mention that, while the existence of solutions in the mass-subcritical case for exterior domains has already been addressed in \cite{ZhZh22NODEA} and \cite{LM_PRSE_23}, to the best of our knowledge, this is the first result dealing with the mass supercritical case in exterior domains.  
We would like to highlight that our assumptions regarding the domain's size or the mass's magnitude are not optimal, and we believe they can be removed. 
However, we will investigate this issue in a subsequent paper. 

\medskip

Before analyzing the assumption on the shape of our domains, let us fix the notations on the Morse index. 

\begin{definition}
\label{def morse}
\begin{itemize}
    \item 
Let $\lambda,\eta\in \R$ be fixed and  $u \in  H^1_0(\Omega)$ be a solution of $-\Delta u + \lambda u = \eta |u|^{p-2} u$,
the  \emph{Morse index of $u$}, is the value 
 \begin{equation}\label{second differential}
 \begin{split}
      i (u):= \sup \Big\{\dim Y \ :&\ Y   \subset H^1_0(\Omega) \ \mbox{such that} \\
      &\int_{\Omega} \left(| \nabla \varphi|^2 + (\lambda-(p-1)\eta|u|^{p-2}) \varphi^2\right)\,dx<0 \quad \forall \varphi \in  Y \setminus \{0\}\Big\}.
      \end{split}
\end{equation}

\item 
Given $u \in S_m$, the  \textit{Morse index} as a constrained critical point of $I_\eta$  in $u$ is the value 
\begin{equation} \label{def: morse} 
   \widetilde{i} (u):= \sup \left\{\dim Y \ :\ Y   \subset T_u S_m \ \mbox{such that} \ D^2 I_\eta(u)\left[\varphi,\varphi  \right] <  0, \ \forall \ \varphi \in Y \setminus \left\{ 0\right\}\right\}
\end{equation}
where 
\begin{equation*}
    D^2 I_\eta(u) \left[ \cdot, \cdot \right] := I_\eta''(u) \left[\cdot, \cdot \right]- \frac{ I_\eta' (u)\left[u\right]}{|u|_2^2} \langle \cdot, \cdot \rangle_{L^2}.
\end{equation*}

\item 
Given $\zeta>0$ and $u \in S_m$, the $\zeta$-\textit{approximate Morse index} of $I_\eta$  in $u$ is the value 
\begin{equation} \label{def: app morse} 
   \widetilde{i}_\zeta(u):= \sup \left\{\dim Y \ :\ Y   \subset T_u S_m \ \mbox{such that} \ D^2 I_\eta(u)\left[\varphi,\varphi  \right] <  - \zeta \Vert \varphi \Vert^2, \ \forall \ \varphi \in Y \setminus \left\{ 0\right\}\right\}.
\end{equation}
 \end{itemize}
 \end{definition}

On the domains, we consider the following assumptions:

 \beq
 \label{NE}\tag{NE$_{\Omega}$}
 \begin{split}
&\Omega \ \mbox{ is smooth and }\ \Delta u+ u^{p-1}= 0, \quad u\in H^1_0(\Omega)\setminus\{0\},\ u\geq 0,\\
& \mbox{ has no solution $u$ with }i(u)\le N+2, 
\end{split}
\eeq
and, for $m>0$, 
 \beq
\label{NEm}
\tag{NE$_{\Omega,m}$}
\begin{split}
&\Omega \ \mbox{ is smooth and }\ \Delta u+ u^{p-1}= 0, \quad u\in H^1_0(\Omega)\setminus\{0\}, \ u\geq 0, \\
& \mbox{ has no solution $u$ with }i(u)\le N+2 \mbox{ and } \|u\|_{L^2}<m.
\end{split}
\eeq
\begin{remark} \label{Nonexistence}
By a Pohozaev identity, one can verify that Hypotheses \eqref{NE} and \eqref{NEm} are satisfied if, for example, the complement of the domain $\Omega$ is assumed to be star-shaped with respect to the origin. 
On the other hand, it is possible to show that under suitable assumptions on $\mathbb{R}^N \setminus \Omega$ or on the exponent $p$, there exist domains for which these two hypotheses may fail. 
For a more detailed analysis, readers may refer to \cite{MR1658565}.
Moreover, for similar non-existence hypotheses in the study of normalized solutions,  we refer, for instance, to Hypothesis $(G3)$ in \cite{MR4701352},  or to Hypothesis $(g2)$ in \cite{cingolani2024} for the case of a more general nonlinearity in $\mathbb{R}^N$.
\end{remark}
 
For every fixed $m>0$, we have the following existence result.

\begin{theorem}
\label{T1}
    Let $m>0$, then there exists $\varrho_m>0$ such that
    \eqref{Pm} has a positive solution $u$ whenever $\Omega$ satisfies  \eqref{NEm} and $\R^N\setminus \Omega\subseteq B_{\varrho_m}(0)$.
    Moreover,   $\lambda\ge 0$ and $\widetilde{i} (u)\le N+1$.
\end{theorem}

In the proof of Theorem \ref{T1}, there are two main points where we use that $\R^N\setminus\Omega$ is small.
The first one is in the construction of the linking structure for the energy functional $I_\eta$, which has to be uniform for $\eta$ in a neighborhood of 1.
Indeed, to get this structure, we take advantage of the mountain pass solution $\widetilde Z_m$ of Problem \eqref{Pm} on all the space $\R^N$.
Anyway, the crucial point where it is used is in the compactness analysis, where we have to work in a suitable compactness range, in the spirit of \cite{BC87ARMA}.

If we consider $\Omega$ fixed, without any assumptions on the size of $\R^N\setminus \Omega$, and make a careful analysis of the asymptotic energy of $\widetilde Z_m$, as $m\to\infty$, we will see that we can recover both the linking structure and the compactness, for large $m$. We point out that this approach is possible because here we are working on the mass super-critical regime  $p > 2 + \frac{4}{N}$ (see \eqref{1725} and \eqref{1726}).
As a consequence, we get the following result.

\begin{theorem}
\label{T2}
   Let $\Omega$ be an exterior domain, then there exists $m_\Omega>0$ such that
    \eqref{Pm} has a positive solution $u_m$ whenever $\Omega$ satisfies  \eqref{NE} and $m>m_\Omega$.
    Moreover,    $\lambda > 0$ and $\widetilde{i} (u_m)\le N+1$.
\end{theorem}

\begin{remark}
    By the proof of Theorem \ref{T2}, we will verify that for every $\bar\varrho>0$ there exists $m_{\bar\varrho}>0$ such that $m_\Omega\le m_{\bar\varrho}$, when $\R^N\setminus\Omega\subseteq B_{\bar\varrho(0)}$.
\end{remark}

The main difficulty in proving Theorem \ref{T2} is to get the energy estimates and the linking structure uniformly for large $m$ (see Lemma \ref{L7}). 

\begin{remark}
Utilizing Remark \ref{rho1} and Lemma \ref{L3.1}, it becomes evident that the assumption concerning smallness in Theorem \ref{T1} can be analyzed with respect to the capacity of $\R^N\setminus\Omega$.
Moreover, the map $m\mapsto \varrho_m$ is an increasing map.

\end{remark}

\begin{remark}
 According to the proofs of Theorems \ref{T1} and \ref{T2}, the assumptions guarantee the existence of a solution to problem \eqref{eq U} for every $\eta\in(1-\eta_\Omega,1+\eta_\Omega)$, where $\eta_\Omega>0$ is suitably chosen. 
 Consequently, for each $\eta\in(1-\eta_\Omega,1+\eta_\Omega)$, employing scaling arguments, we derive a solution $\bar u_\eta$ for problems analogous to \eqref{Pm}. 
 However, concerning the solution $\bar u_\eta$, we have not manageable information about either the domain $\Omega_\eta$ or the mass $\|\bar u_\eta\|_{L^2(\Omega_\eta)}$, both of which depend on the parameters $\eta$ and $\lambda_\eta$. 
\end{remark}

\medskip

The paper is organized as follows: in Section 1 we introduce the notations and some tools, 
then we fix the mass $m$ and highlight the linking structure that produces bounded Palais-Smale sequences. 
In Section 3 we analyse some features of the weak limits of the PS sequences, and in Section 4 we show that they solve some approximating problems.
Finally, in Section 5, we prove Theorem \ref{T1} by performing a blow-up analysis of the approximate solutions.
In Section 6, Theorem \ref{T2} is proved.

\newpage


\section{Notations, linking structure and bounded PS sequences}
\label{S3}


{\bf Notations}:
{\small
\begin{itemize}

\item 
$B_r(y)$ denotes the open ball of radius $r>0$ and center $y\in\R^N$, and we write $B_r$ when $y=0$.

\item For any subset $\cD\subset \R^N$, we denote by $\cD^c=\R^N\setminus \Omega$ its complementary set,   and by $|\cD|$ its Lebesgue measure.
\item
For $\Omega\subseteq \R^N$ and $1\leq p\le \infty$,  $L^p(\Omega)$ and $H^1_0(\Omega)$ are the usual Lebesgue and Sobolev spaces, with norm
$|u|_{p}=\left(\int |u|^pdx\right)^{1/p}$, $1\le p<+\infty$, $|u|_\infty =\esssup_{\Omega}|u|$, and $\|u\|=\left(\int (|\D u|^2+u^2)\,dx\right)^{1/2}$, respectively. 

Here, we agree $H^1_0(\Omega)\subset H^1(\R^N)$, setting $u\equiv 0$ in $\R^N\setminus\Omega$, $\forall u\in H^1_0(\Omega)$, and write $\int=\int_{\R^N}$.

\item For any measurable set $\cD\subset \R^N$, we denote by $|\cD|$ its Lebesgue measure.

\item $c,\bar c,\tilde c,\ldots$ are generic constants that can vary from line to line.

\end{itemize}
}    

It has been proved by Kwong in \cite{MR0969899} that there exists a function $U \in C^2(\R^N)$, unique up to translations, radially symmetric and decreasing, that solves
\begin{equation} \label{Pinf}
\tag{$P_\infty$}
\begin{cases} -\Delta U + U = U^{p-1} & \text{ in } \R^N, \\ 
U>0 \\
U(0)=| U |_\infty.
\end{cases} 
\end{equation}

The function $U$ exhibits the following asymptotic behavior: there exists a constant \(c_1 > 0\) for which
\begin{equation} \label{eq1}
U(r) e^{r} r^{\frac{N-1}{2}} \to c_1 \text{ as } r \to \infty,
\end{equation}
as well as
\begin{equation} \label{eq2}
U'(r) e^{r} r^{\frac{N-1}{2}} \to -c_1 \text{ as } r \to \infty.
\end{equation}

Let $\eta>0$, the scaled function $U_\eta:=\eta^{-\frac{1}{p-2}}U$ solves
\begin{equation} \label{Pinfeta}
\tag{$P_\infty^\eta$}
\begin{cases} -\Delta U_\eta + U_\eta = \eta U_\eta^{p-1} & \text{ in } \R^N, \\ 
U_\eta>0 \\
U_\eta(0)=|U_\eta|_\infty.
\end{cases} 
\end{equation}
Moreover, for $\eta>0$ and $m>0$, and choosing $\mu=\mu(\eta,m)$ as the solution of 
\begin{equation}
\label{1725}
\frac{\mu^{p-\left(2+\frac 4N\right)} }{\eta^{\frac{2}{N}}}=\left(\frac{ m}{m_0}\right)^{\frac2N(p-2)}, \quad \quad m_0:=|U|_2,
\end{equation}
it is easy to verify that
\begin{equation}
\label{1726}
\widetilde Z_{m,\eta}(x):= \mu^{-\frac{2}{p-2}} U_\eta \left(\frac{x}{\mu}\right)
\end{equation}
solves the prescribed mass problem
\begin{equation} \label{Pminfeta}
\tag{$P_{m,\infty}^\eta$}
\begin{cases} -\Delta\widetilde  Z_{m,\eta} + \lambda_m\widetilde  Z_{m,\eta} = \eta\widetilde  Z_{m,\eta}^{p-1} & \text{ in } \R^N, \\ 
|Z_{m,\eta}|_2=m
\end{cases} 
\end{equation}
where $\lambda_m=\mu^{-2}$. 
\begin{remark}
    \label{RJJ}
The solution $\widetilde Z_{m,\eta}$ of \eqref{Pminfeta} can be obtained as a mountain pass-type critical point of the functionale $I_\eta$ constrained on $\widetilde{S}_m$ (see, for example, \cite{JJ97}).
\end{remark}

We will denote the critical level related to \eqref{Pminfeta} as 
\begin{equation}
\label{1728}
    c_{m,\eta}=  I_\eta \left( \widetilde  Z_{m,\eta} \right). 
\end{equation}
In the notation we will drop the dependence of $\eta$ when $\eta=1$, writing simply $c_m:=c_{m,1}$, $I:=I_\eta$, $\widetilde Z_{m}:=\widetilde Z_{m,1}$, \ldots . 

\begin{remark}
Taking into account \eqref{1725}--\eqref{1728}, direct computations show that the following energy relations hold:
 \begin{equation}
 \label{eq_decr}
     c_{m,\eta}=\eta^{-\frac{4}{N}\left[p-\left(2+\frac{4}{N}\right)\right]^{-1}} c_m\quad\Longrightarrow\quad  c_{m,\eta}>c_{m},\quad  \forall\eta\in(0,1),\ \mbox{ and } c_{m,\eta}\to c_{m}\ \mbox{ as }\eta\to 1.
    \end{equation}
\end{remark}
 
\medskip

For $h \in \R$ and $ u \in H^1( \R^N)$, we define the scaling
\begin{equation*}
h \ast u(x) := e^{\frac{N}{2} h} u(e^h x).
\end{equation*}
Notice that it is immediate to verify that this scaling preserves the $ L^2 $-norm. i.e. \( \|h \ast u\|_2 = \|u\|_2 \) for all \( h \in \mathbb{R} \) and that
\beq\label{1638}
c_{m,\eta}=I_\eta(\widetilde Z_{m,\eta})=I_\eta(0\ast \widetilde Z_{m,\eta}),\qquad \forall \eta>0,
\eeq
$$
\lim_{h\to -\infty}I_\eta (h\ast \widetilde Z_{m})=0,\qquad  \lim_{h\to +\infty}I_\eta(h\ast\widetilde  Z_{m})=-\infty,\ \mbox{ uniformly in }\eta\in[1/2,1].
$$

Then, we can fix $\bar h>0$ such that 
\beq
\label{1847}
 I_\eta(-\bar h\ast\widetilde  Z_{m})<c_{m},\qquad  I_\eta(\bar h\ast\widetilde  Z_{m})<0\qquad \forall  \eta\in[1/2,1].
\eeq
Now, we consider a smooth radial cut-off function $\vartheta \ : \ \R^n \to \left[0,1\right]$ such that
\begin{equation*}
\vartheta(x)= \begin{cases} 0 & 0 \leq |x| \leq 1 \\
   1 & |x| > 2 \\
   |\nabla \vartheta| \leq { 2}         
    \end{cases}
\end{equation*}
and, for $\varrho>0$, we define the functions 
\begin{equation}
\label{1624}
\widetilde  Z_{m  }{[y,h] }=h \ast\widetilde  Z_{m } (\cdot-y) ,\quad  
{Z}_{m,\varrho   }  [y,h] = m \ \frac{\displaystyle \vartheta\left({x}/{\varrho}\right)   \widetilde Z_{m }{[y,h] }}{\displaystyle \left| \vartheta\left({x}/{\varrho}\right) \widetilde Z_{m  }{[y,h] }\right|_2},\quad y\in\R^N,\ h\in\R.
\end{equation}
We would like to emphasize that if $\R^N \setminus \Omega \subset B_\varrho$, then  $ {Z}_{m ,\varrho}[ {y},h] \in S_m$.
From now on, we will denote the scaled cut-off as $\vartheta_\varrho:= \vartheta(x / \varrho)$. 

For \(\bar  R > 0 \)  which will be determined later, we set
\begin{equation}
\label{EQ}
Q := \overline{B}_{\bar R} \times  [-\bar h, \bar h] \subset \mathbb{R}^N \times \mathbb{R}
\end{equation}
and, for $m,\varrho >0$, we define
\[
\Gamma_{m,\varrho} := \{ \gamma : Q \to S_m \ :\  \gamma \text{ continuous}, \, \gamma(y, h) =   {Z}_{m,\varrho}  [y,h]  \text{ for all } (y, h) \in \partial Q \}.
\]
When $\Omega^c\subset B_{\varrho}$, $\varrho$ is small and $\eta$ is near 1, our aim  will be to find a solution of \eqref{eq U}  whose candidate critical level is
\beq
\label{1453}
c^\eta_{\Omega,m} := \inf_{\gamma \in \Gamma_{m,\varrho}} \max_{(y,   h) \in Q} I_\eta (\gamma(y, h)).
\eeq

In the following lemma, we collect some properties that can be proved by standard computations (see, for instance, \cite[Lemma 4.1]{BC87ARMA}).

\begin{lemma}
\label{L3.1}
Let $\varrho>0$, $\Omega^c\subset B_{\varrho}$, and $\bar m>0$, then
\begin{itemize} 
    \item[(i)] $(y,h)\mapsto {Z}_{m,\varrho}[y,h]$  is a continuous map $\R^N\times \R\to H^1_0(\Omega)$, for every \( m >0 \).
    \item[(ii)] \(  {Z}_{m,\varrho}[y,h] \to \widetilde Z_{m}[y,h] \), as $\varrho\to 0$,  strongly in \( H^1(\R^N) \),   uniformly in \( y\in\R^N \), $h\leq \const$ and 
     $m\geq  \bar m$. 
    \item[(iii)]      \(  {Z}_{m,\varrho}[y,h] \to \widetilde Z_{m}[y,h] \), as $|y|\to\infty$,  strongly in \( H^1(\R^N) \),    uniformly in { $h\in\R$}, $ \Omega^c\subset B_\varrho$  and 
   $m\geq  \bar m$. 
\end{itemize}
\end{lemma}

\begin{remark}
\label{rho1}
Observe that by Lemma \ref{L3.1} $(ii)$ and Remark \ref{RJJ} we can fix $\varrho_1>0$ such that if $ \Omega^c\subset B_\varrho$, with $\varrho<\varrho_1$, then
 we can use ${Z}_{m,\varrho}[y,-\bar h]$ and ${Z}_{m,\varrho}[y,\bar h]$ as endpoints for the paths in the mountain pass characterization of $c_m$.
 Moreover,
   $$
    \max_{t\in[0,1]} I\left((1-t){Z}_{m,\varrho}[y,-\bar h]+t\widetilde{Z}_{m,\varrho}[y,-\bar h]\right)<c_m,
    $$
    $$
        \max_{t\in[0,1]} I\left((1-t){Z}_{m,\varrho}[y,\bar h]+t\widetilde{Z}_{m,\varrho}[y,\bar h]\right)<c_m.
    $$
\end{remark}

\medskip

Now we recall the notion of barycenter of a function \( u \in H^1(\mathbb{R}^N) \setminus \{0\} \) which has been introduced in \cite{MR1989833} (see also \cite{BaWe05}).
Setting
\[
\tilde{u}(x) = \frac{1}{|B_1(0)|} \int_{B_1(x)} |u(y)| \, dy,
\]
we observe that \(\tilde{u}\) is bounded, continuous, and vanish at infinity, so the function
\[
\hat{u}(x) = \left( \tilde{u}(x) - \frac{1}{2} \max \tilde{u} \right)^+
\]
is well-defined, continuous, and has compact support. Therefore we can define \( \beta : H^1(\mathbb{R}^N) \setminus \{0\} \to \mathbb{R}^N \) as
\[
\beta(u) = \frac{1}{|\hat{u}|_1} \int_{\mathbb{R}^N} \hat{u}(x) \, x \, dx.
\]
The map \( \beta \) is well defined, because \(\hat{u}\) has compact support, and it is not difficult to verify that it satisfies the following properties:
\begin{itemize}
    \item \( \beta \) is continuous in \( H^1(\mathbb{R}^N) \setminus \{0\} \);
    \item if \( u \) is a radial function, then \( \beta(u) = 0 \);
    \item \( \beta(tu) = \beta(u) \) for all \( t \neq 0 \) and for all \( u \in H^1(\mathbb{R}^N) \setminus \{0\} \);
    \item setting \( u_z(x) = u(x - z) \) for \( z \in \mathbb{R}^N \) and \( u \in H^1(\mathbb{R}^N) \setminus \{0\} \), we have \( \beta(u_z) = \beta(u) + z \).
\end{itemize}
Now, define
\[
\mathcal{D} := \{ D \subseteq S_m : D \text{ is compact, connected, } (-\bar  h) \ast Z_m ,  \bar h \ast Z_m \subseteq D \},
\]
\[
\mathcal{\widetilde{D}} := \{ D \subseteq \widetilde{S}_m : D \text{ is compact, connected, } { (-\bar h)} \ast Z_m , {\bar h} \ast Z_m \subseteq D \},
\]
\[
\mathcal{D}_0 := \{ D \in \mathcal{D} : \beta(u) = 0 \text{ for all } u \in D \},
\qquad 
\mathcal{\widetilde{D}}_0 := \{ D \in \mathcal{\widetilde{D}} : \beta(u) = 0 \text{ for all } u \in D \}.
\]

We also set
\[
\ell_m := \inf_{D \in \mathcal{\widetilde{D}}} \max_{u \in D} I (u)
\]
\[
\ell_{m}^0 := \inf_{D \in \mathcal{\widetilde{D}}_0} \max_{u \in D} I (u).
\]
\begin{lemma} \label{lemma4} 
Let $ \Omega^c\subset B_{\varrho_1}$, then
$$
\ell_{m}^0=\ell_m=c_m.
$$
\end{lemma}
\begin{proof}
    The result follows similarly to \cite[Lemma 3.2]{MR4304693}, with the additional consideration of Remark \ref{rho1}. 
    In particular, to show that $c_m\geq \ell_{m}^0$, we consider the energy of the path
$$
\gamma(t):=\ \left\{\begin{array}{ll}
(1-3t)\, {Z}_{m}[y,-\bar h]+3t\, \widetilde{Z}_{m}[y,-\bar h] &t\in[0,1/3]\\
\widetilde {Z}_{m}[y,3(2t-1)\, \bar h] &t\in[1/3,2/3]\\
3(1-t)\, \widetilde {Z}_{m}[y,\bar h]+(3t-2)\, {Z}_{m}[y,\bar h] &t\in[2/3,1].
\end{array}
\right.
$$
\end{proof}

\begin{proposition}
\label{P3.4}
For every exterior domain $\Omega$, we have that 
\begin{equation*}
    L_{\Omega,m}:= \inf_{D \in \mathcal{D}_0} \max_{u \in D} I(u) > c_m.
\end{equation*}
\end{proposition}
\begin{proof}
    First of all notice that $L_{\Omega,m} \geq c_{m}$ by Lemma \ref{lemma4}, because   $\mathcal{D}_0 \subset \mathcal{\widetilde{D}}_0$.
     
    Now, assume by contradiction that there exists a sequence $D_n$ in $ \mathcal{D}_0$ such that
    \begin{equation*} 
        \max_{u \in D_n} I(u) \to c_m.
    \end{equation*}
    
    At this point, arguing exactly as in \cite[Lemma 3.4]{MR4304693}, it is possible to deduce the existence of a   sequence $(v_n)_n $, with $v_n\in D_n$, such that $v_n \to \pm Z_m$ in $H^1_0(\Omega)$. 
    However, this is not possible since it would imply $Z_m =0 $ in $\R^N \setminus \Omega$.
\end{proof}

\begin{remark}
\label{Rem_buco}
Clearly, if $B_{R_i}\subseteq\Omega^c \subseteq B_{R_e}$ it turns out that
$$
 L_{(B_{R_i})^c,m} \le L_{\Omega,m} \le L_{(B_{R_e})^c,m}.
$$
Then, it is readily seen that 
$$
 L_{\Omega,m}\to c_m\qquad\mbox{ as } \Omega^c\subset B_\varrho \mbox{ with }\varrho\to 0.
$$
 
\end{remark}

\begin{proposition} \label{prop3}
    $L_{\Omega,m}\le c_{\Omega,m}$.
\end{proposition}
\begin{proof}
By symmetric reasons, we see that $\beta(Z_m[0, h])=0$, $\forall h\in\R$, and for every $(y,h)\in Q$, with $y\neq 0$, there exists $b(|y|,h)>0$ such that $\beta(Z_m[y, h])=b(|y|,h)\, y$.
    Then, the proof goes on arguing as the proof of  \cite[Proposition 3.5]{MR4304693}
\end{proof}

\begin{proposition} \label{prop2}
For any  $\varepsilon > 0$  it is possible to find  $\bar{R} > 0$,   $\eta_\e \in(0,1)$ and ${\varrho_\e}\in (0,{\varrho_1}]$ (see Remark \ref{rho1}) such that for  $Q$ as in \eqref{EQ}   and $\Omega^c\subset B_{\varrho_\e}$ the following holds 
\begin{equation*}
\max_{(y, h) \in \partial Q} I_\eta({Z}_{m}[y,h]) < c_{m,\eta} + \epsilon,\qquad \forall \eta\in (\eta_\e,1].
\end{equation*}
\end{proposition}

\begin{proof}
If $\Omega^c\subset B_{\varrho_1}$, then, by using Lemma \ref{L3.1} $(iii)$, we see 
$$
\lim_{R\to\infty}\sup\{I({Z}_{m}[y,h])\ :\ |y|=R,\ |h|\le \bar h\}=\max_{h\in[-\bar h,\bar h]}I(\widetilde Z_m[0,h])=  
I(\widetilde Z_m[0,0])=c_m.
$$
So, we obtain the existence of $\bar R>0$ and $\hat\eta\in(0,1)$ such that
\beq
\label{1759}
\max\{I_\eta ({Z}_{m}[y,h])\ : \ |y|=\bar R,\ |h|\le\bar h\}< c_{m,\eta} + \epsilon, \qquad \forall\eta\in(\hat\eta,1],
\eeq
because of
\begin{equation}
\label{1806}
    I_\eta(u)= I(u) +\frac{1-\eta}{p} \int | u |^p \, dx,\qquad \forall u\in H^1(\R^N).
\end{equation}

\smallskip

Now, by \eqref{1847}, Lemma \ref{L3.1} $(ii)$ and \eqref{1806} there exist $\varrho_\e\in (0,\varrho_1]$ and $\eta_\e\in(0, \hat\eta]$ such that if $\Omega^c\subset B_{\varrho_\e}$ then
\beq
\label{1817}
\max\{I_\eta ({Z}_{m}[y,h])\ : \ |y|\leq\bar R,\ |h|=\bar h\}< c_{m,\eta} + \epsilon, \qquad \forall\eta\in(\eta_\e,1].
\eeq

\smallskip

Then, the statement follows from \eqref{1759} and \eqref{1817}
 \end{proof}
 
 \begin{proposition}
\label{P3.8}
There exists $\varrho_2\in(0,\varrho_1]$ such that if $\Omega^c\subset B_{\varrho_2}$, then we can find $\eta_{\varrho_2}$ such that for a.e. $\eta \in \left( \eta_{\varrho_2},1\right]$ there is a bounded Palais-Smale sequence $(u_n)_n$ for $I_\eta$ at the level $c_{\Omega,m}^\eta$, with $u_n\geq 0$, $\forall n\in\N$.
Moreover, there exists $(\zeta_n)_n$ in $\R$  with $\zeta_n\searrow 0$ such that the following estimate of the approximate Morse index holds
 \beq
 \label{1827}
 \widetilde{i}_{\zeta_n}(u_n)\leq N+1.
 \eeq
\end{proposition}

\begin{proof}
    By \eqref{eq_decr} and Proposition \ref{P3.4} we obtain $L_{\Omega,m} > c_{m,\eta} > c_m$, if $\eta$ is sufficiently close to $1$.
    
    Now, by taking $\varepsilon$ small enough in Proposition \ref{prop2}, it is possible to find a $Q$, independent of $\eta$ near 1, such that if $\Omega^c\subset B_{\varrho_2}$, for $\varrho_2$ small,
    \begin{equation*}
        \sup_{(y, h) \in \partial Q} I_\eta( {Z}_{m}[y,h]) < L_{\Omega,m} \leq c_{\Omega,m} \leq c_{\Omega,m}^\eta,
    \end{equation*}
where the the second inequality follows from Proposition \ref{prop3}, and the last inequality is held by \eqref{1806}.

At this point, the proof of the proposition concludes applying \cite[Theorem 1.10]{MR4748624}, also taking into account \cite[Remarks 1.8 and 1.11]{MR4748624}.
\end{proof}
   

\section{Analysis of the weak limit}


We are using the following

\begin{notat}
    \label{ueta}
Let $\Omega^c\subset B_{\varrho_2}$ and $\eta\in(\eta_\Omega,1]$ as in Proposition \ref{P3.8}, 
we denote by
\beq
\label{1052}
A_\Omega:=\{ \eta\in (\eta_{\varrho_2},1)\  :\   \mbox{ there exists a bounded Palais Smale sequence, $(u_n)_n$ at the level }c^\eta_{\Omega,m}\}.
\eeq
Moreover, we fix $\eta\in A_\Omega$ and then call $(u_n)_n$ a corresponding bounded Palais Smale sequence at level $c^\eta_{\Omega,m}$, with $u_n\geq 0$.
We assume that $u_n\rightharpoonup u_\eta$, for a suitable $u_\eta\in H^1_0(\Omega)$.
\end{notat}

\begin{proposition} \label{prop1}
  Within the context of Notation \ref{ueta}, the function $u_\eta$ satisfies 
  \begin{equation} 
  \label{1149}
  -\Delta u_\eta + \lambda_\eta u_\eta= \eta u_\eta^{p-1},
    \end{equation}
   for some $\lambda_\eta \geq  0$. 
\end{proposition}

Before proving Proposition \ref{prop1}, we state two lemmas. 
The first generalizes \cite[Lemma 3.2]{MR4601303} to the case of suitable subspaces of $H^1_0(\Omega)$ of any given dimension $k\in\N$.

\begin{lemma} \label{lemma3}
    For any $\lambda <0$ and  $k \in \N$, there exists a subspace $Y \subset H_0^1(\Omega)$ with $\dim Y=k$ such that
    \begin{equation*}
        \int_\Omega |\nabla u|^2 \, dx + \lambda \int_\Omega |u|^2 \, dx \leq \frac{\lambda}{2} \Vert u \Vert^2 \quad \forall \ u \in Y.
    \end{equation*}
\end{lemma}
\begin{proof}
  Let us fix $\varphi \in C_0^\infty(\R^N)$ such that $\supp \varphi \subset  B_2(0)$ and $|\varphi |_2 =1$. 
  Set $\alpha=|\nabla \varphi|_2$ and select  $h<0$ small enough to get
  \begin{equation} \label{eq5}
      \frac{e^{2h} \alpha+ \lambda}{e^{2h} \alpha+ 1} \leq \frac{\lambda}{2}.
  \end{equation}
  We now define the functions
  \begin{equation*}
      \varphi_j (x) := h \ast \varphi (x-r_j e_1) \quad j=1,...,k
  \end{equation*}
  where $e_1=(1,0,\ldots,0)\in\R^N$ and  $r_j$, $j=1,\ldots,k$, are such that the supports of $\varphi_j$ are disjoint and $\supp \varphi_j \subset \Omega$.
    Notice that this is possible being $ \Omega^c$ bounded. 
    
    We define $Y=\spann\{\varphi_j\, :\, j=1,\ldots,k\}$ and we claim that $Y$ verifies our statement.
    Indeed, let $u=\sum_{j=1}^k \tau_j \varphi_j\in Y$, then
  \begin{align} \label{eq6}
      \int_\Omega |\nabla u|^2 \, dx + \lambda \int_\Omega |u|^2 \, dx &= \sum_{j=1}^k { \tau_j^2} \left( e^{2h} \int_\Omega |\nabla { \varphi}|^2 \, dx + \lambda \int_\Omega |{ \varphi}|^2 \, dx\right) \notag \\
      & = \left( { e^{2h}}  \alpha + \lambda \right)\sum_{j=1}^k { \tau_j^2} .
  \end{align}
  With a similar computation, it turns out that
  \begin{equation} \label{eq7}
      \Vert u \Vert^2= \left( { e^{2h}}  \alpha + 1 \right)\sum_{j=1}^k  { \tau_j^2}. 
  \end{equation}
  Coupling \eqref{eq6} and \eqref{eq7}, we get
  \begin{equation*}
  \frac{\int_\Omega |\nabla u|^2 \, dx + \lambda \int_\Omega |u|^2 \, dx}{\Vert u \Vert^2}= \frac{e^{2h} \alpha + \lambda }{e^{2h}\alpha + 1 } \leq \frac{\lambda}{2},
  \end{equation*}
  where we used \eqref{eq5} in the last inequality. 
\end{proof}

\begin{lemma} \label{lemma2} 
    Let $(u_n)_n \subset S_m$, $(\lambda_n)_n \subset \R$, $(\zeta_n)_n \subset \R_+$ with $\zeta_n \searrow 0$, $k \in \N$   and $\eta>0$. 
    Suppose that the following facts hold:
    \begin{description}
        \item[$(i)$] 
        if $Y_n \subset H^1_0(\Omega)$ is a subspace such that
        \begin{equation*}
            \left(I_\eta \right)''(u_n)\left[\varphi,\varphi \right] + \lambda_n \Vert \varphi \Vert_2^2 <  -\zeta_n \Vert \varphi \Vert^2 \quad \forall \ \varphi \in Y_n,
        \end{equation*}
        then $\dim Y_n \leq k-1$ for $n$ large enough;
        \item[$(ii)$] 
        there exist $\lambda \in \R$, a subspace $Y \subset H^1_0(\Omega)$, with $\dim Y \geq k$, and $a>0$ such that for $n$ large enough we have 
        \begin{equation*}
            \left(I_\eta \right)''(u_n)\left[\varphi,\varphi \right] + \lambda \Vert \varphi \Vert_2^2 < -a \Vert \varphi \Vert^2 \quad \forall \ \varphi \in Y;
        \end{equation*}
    \end{description}
    Then we have $\lambda_n  >  \lambda$ for $n$ large enough. 
    In particular, if the inequality in $(ii)$ holds for any $\lambda <0$, we have $\liminf_{n \to \infty} \lambda_n \geq 0$.
\end{lemma}
\begin{proof}
    Arguing by contradiction, we suppose that there exists a subsequence, still denoted with $\lambda_n$, such that $\lambda_n \leq \lambda$. 
    If so, we would have
    \begin{align*}
        \left(I_\eta \right)''(u_n)\left[\varphi,\varphi \right] + \lambda_n \Vert \varphi \Vert_2^2 & =\left(I_\eta \right)''(u_n)\left[\varphi,\varphi \right] + \lambda \Vert \varphi \Vert_2^2+\left( \lambda_n-\lambda\right)  \Vert \varphi \Vert_2^2 \\
        & \leq -a \Vert \varphi \Vert +\left( \lambda_n-\lambda\right)  \Vert \varphi \Vert_2^2
    \end{align*}
    for all $\varphi \in Y \setminus \left\{0 \right\}$. 
    Now, since $\zeta_n \searrow 0$, we finally have $\zeta_n < a$. 
    Using this and the contradiction hypothesis in the inequality above, we obtain
    \begin{equation*}
        \left(I_\eta \right)''(u_n)\left[\varphi,\varphi \right] + \lambda_n \Vert \varphi \Vert_2^2 <-\zeta_n \Vert \varphi \Vert, 
    \end{equation*}
    for $n$ large enough, for all $\varphi \in Y \setminus \left\{0 \right\}$ with $\dim Y \geq k$. 
    However, this is not admissible because it would violate $(i)$. 
\end{proof}

\medskip

{\em Proof of Proposition \ref{prop1}.}\ 
 Applying \cite[Lemma 3]{MR695536}, we obtain
\begin{equation} 
\label{1356}    I'_\eta(u_{n}) - \lambda_{n} u_{n} \to 0 \quad \text{in} \ \left(H^1_0(\Omega)\right)^*,
\end{equation}
where 
\begin{equation*}
    \lambda_n := \frac{I'_\eta(u_{n}) u_{n}}{ m^2}.
\end{equation*}
We note that \(\lambda_n\) is also bounded, so we can assume \(\lambda_n \to \lambda_\eta\). 
As a consequence, \(u_\eta\) weakly solves
\begin{equation*}
    -\Delta u_\eta + \lambda_\eta u_\eta = \eta \, u_\eta^{p-1}.
\end{equation*}

Next, since \(\widetilde{i}_{\zeta_n}(u_{n}) \leq N+1\) and \(\codim T_{u_{n}} S_m = 1\), the inequality in \((i)\) of Lemma \ref{lemma2} is satisfied for subspaces of \(H^1_0(\Omega)\) with dimension not greater than \(N+2\). 
Furthermore, Lemma \ref{lemma3} implies that the condition \((ii)\) in Lemma \ref{lemma2} is also verified for any $\lambda<0$. 
Therefore, by Lemma \ref{lemma2}, we conclude \(\lambda_\eta \geq 0\).
\qed

 \smallskip

\begin{proposition}
\label{Pmorse}
The solution $u_\eta$ of \eqref{1149} provided by Proposition \ref{prop1}, if not identically zero,  verifies
$$
i(u_\eta)\le N+2.
$$
\end{proposition}

\begin{proof} 
Throughout this proof, we let $\alpha:=|u_\eta|_2$.
For all $w \in H_0^1(\Omega)$,  
\begin{equation} \label{eq26}
D^2 I_{\eta}(u_{\eta})\left[w,w\right] = I''_{\eta}(u_{\eta})\left[w,w\right] + \lambda_{\eta} \langle w,w\rangle_{L^2} = \int_{\Omega} \left[ |\nabla w|^2 + \left( \lambda_{\eta} -\eta(p-1) |u_{\eta}|^{p-2}\right) w^2\right]  \, dx
\end{equation}  
where
\begin{equation*}       
   \lambda_\eta :=  -\frac{1}{\alpha^2}I'_\eta(u_\eta)\left[u_\eta   \right] =  - \lim_{n \to \infty}\frac{1}{m^2}I'_\eta(u_n) \left[u_n\right],
   \end{equation*}
   
It is sufficient to show that $\widetilde i(u_\eta)\le N+1$, because $T_{u_\eta} S_{\alpha}$ has codimension 1.
Arguing by contradiction, let us assume that there exists a subspace $Y \subset T_{u_\eta}S_\alpha$ with $\dim Y = N+2$ such that  
\begin{equation*}
D^2 I_{\eta}(u_{\eta})(w,w)  < 0, \quad \mbox{for all} \ w \in Y \backslash \{0\}. 
\end{equation*}
Since $Y$ has finite dimension, there is a constant $\beta > 0$ such that  
\begin{equation*}  
D^2 I_{\eta}(u_{\eta})\frac{(w,w)}{\|w\|^2} < - \beta ,  \quad \mbox{for all } w \in Y\backslash \{0\}.  
\end{equation*}  
Now, observe that $\{|u_n|^{p-2}\}$ is bounded in $L^{N/2}(\Omega)$, because $p\in\left(2+\tfrac4N,\tfrac{2N}{N-2}\right)$.
Then, since $\dim Y<\infty$, 
$$
D^2 I_{\eta}(u_n)\frac{(w,w)}{\|w\|^2} \longrightarrow D^2 I_{\eta}(u_\eta)\frac{(w,w)}{\|w\|^2},\ \mbox{ as } 
n\to\infty,\quad\mbox{ uniformly in }w \in Y\backslash \{0\},
$$
up to a subsequence.
Then, for large $n$,
\beq
\label{10135}
D^2 I_{\eta}(u_n)(w,w) < - \frac{\beta}{2} \|w\|^2 \quad \mbox{for all } w \in Y.
\eeq
Classical reasoning indicates that, for sufficiently large $n$,  $\dim Y=\dim Y_n$, where $Y_n$ represents the projection of $Y$ on $T_{u_n}S_m$, and that equation \eqref{10135} is satisfied when substituting $Y_n$ for $Y$. 
Nevertheless, given that we have assumed $\dim Y = N+2$, this is in contradiction with equation \eqref{1827} in Theorem \ref{Th.eta}, concluding the proof.
\end{proof}


\section{Solving approximating problem}


In this section, we show that the function $u_\eta$, introduced in the Notation \ref{ueta}, is a solution to the approximate problem \eqref{eq U}.

\begin{proposition}
\label{P5.1}
 Assume that  \eqref{NEm} holds, let $\eta\in A_\Omega$ and set 
  \begin{equation*} 
  \lambda_n:= I'_\eta (u_n)u_n ,
    \end{equation*}
    then $\lambda_n\to\lambda > 0$, up to a subsequence.
\end{proposition}
\begin{proof}
     Being $(u_n)_n$ bounded, it is immediate to see that $(\lambda_n)_n$ is bounded, so we can assume $\lambda_n \to \lambda$, up to a subsequence, for some $\lambda \in \R$.
    Moreover, from Proposition \ref{prop1}, we also have $\lambda \geq 0$.   
    Assume by contradiction that $\lambda=0$.
    
   First, we observe that $\lambda =0$ and \eqref{1356} imply 
    \begin{equation} \label{eq4}
        I'_\eta(u_n)=I'_\eta(u_n)-\lambda_n u_n \to 0 \quad \text{in} \ \left(H^1_0(\Omega)\right)^*.
    \end{equation}
    Then, we consider 
    \begin{equation*}
        \alpha:= \limsup_{n \to \infty } \sup_{z \in \R^N} \int_{B_1(z)} |u_n|^2 \, dx
    \end{equation*}
and distinguish two cases $\alpha=0$ and $\alpha>0$, showing that both are not possible.

\smallskip

If $(u_n)_n$ is vanishing, that is $\alpha=0$, then invoking \cite[Lemma I.1]{MR0778974} we get $u_n \to 0$ in $L^q(\R^N)$ for all $q \in (2,2^*)$.

From this, testing \eqref{eq4} with $u_n$, we can deduce that
    \begin{equation*}
        \int_\Omega |\nabla u_n |^2 \, dx + \lambda_n\int_\Omega u_n^2 \, dx \to 0 
    \end{equation*}
as $n \to \infty$. 
Since $\lambda_n \to 0$, we see that $\int_\Omega |\nabla u_n |^2 \, dx\to 0$ and, recalling that $(u_n)_n$ is a Palais-Smale sequence for $I_\eta$ at level $c_{m,\eta}$, we get $I_\eta(u_n)\to 0$, in contradiction to $I_\eta(u_n)\to c_{m,\eta}>0$.

\smallskip

Assume now $\alpha >0$. 
Then, there exists a sequence $(z_n)_n \subset \R^N$ such that $v_n:= u_n(\cdot+y_n) \rightharpoonup v$ in $H^1(\R^N)$ for some $v \in H^1(\R^N) \setminus \left\{ 0\right\}$. 

On the one hand, if $|z_n| \to \infty$, we observe that 
    \begin{equation*}
        o(1)= I'_\eta(u_n)-\lambda_n u_n= I'_\eta(v_n).
    \end{equation*}
    Moreover, for all $\varphi \in C_0^\infty(\R^N)$, we have that eventually $\supp \varphi \subset \Omega_n$ where 
    \begin{equation*}
        \Omega_n := \left\{ x \in \R^N \ :\  x+ z_n \in \Omega \right\}.
    \end{equation*}
    Thus, $v$ weakly solves
    \begin{equation*}
        -\Delta v = \eta v^{p-1} \quad \text{in} \qquad v\in H^1(\R^N),
    \end{equation*}
but this problem does not admit positive nontrivial solutions as a consequence of the Pohozaev identity (see, f.i., \cite[Proposition 1]{BeLi83ARMA1}).
 
On the other hand, if $|z_n|$ remains bounded, we have $u_n \rightharpoonup v$ in $H^1_0(\Omega)$ for a suitable $v\not\equiv 0$, with $|v|_2\le m$ and $v\geq 0$ because $u_n\geq 0$.
Then, by \eqref{eq4} we obtain a nonnegative weak nontrivial solution $v$ of 
\begin{equation*} \Delta v + \eta v^{p-1}=0 \quad \text{in} \ \Omega,
    \end{equation*}
 with $i(v)\le N+2$ by Proposition \ref{Pmorse}.  
But, in such a case, $u=\eta^{\frac{1}{p-2}}v$ verifies $0<|u|_2<m$ and solves $\Delta u + u^{p-1}=0$    in $\Omega$, contrary to assumption \eqref{NEm}.
    
   \end{proof}

\begin{lemma}
    \label{LperPS}
    For every $m>0$, there exist $\varrho_{ps}\in(0,\varrho_2)$ and $\eta_{ps}\in (0,1)$ such that $c^\eta_{\Omega,m}<\min\left\{2,1+\tfrac2N\right\} c_{m}\le\min\left\{2,1+\tfrac2N\right\} c_{m,\eta}$ whenever $\Omega^c\subseteq B_{\varrho_{ps}}$ and $\eta\in(\eta_{ps},1]$.   
\end{lemma}

The proof is a direct consequence of the definition of $C^\eta_{\Omega,m}$, of \eqref{eq_decr}, and of Lemma \ref{L3.1}.
In fact, it suffices to consider in \eqref{1453} the trivial map $\gamma (y,h)={Z}_{m,\varrho_{ps}}  [y,h]$ for all $(y, h) \in  Q$, for $\varrho_{ps}$ sufficiently small.

   \begin{theorem}
   \label{Th.eta}
    Let $\varrho_{ps}$ and $\eta_{ps}$ be as in Lemma \ref{LperPS}, if  $\Omega^c\subseteq B_{\varrho_{ps}}$  and $\eta\in A_\Omega\cap (\eta_{ps},1]$  (see \eqref{1052}),  then $u_\eta$ is a critical point for $I_\eta$ constrained on $S_m$, at the minmax level  $c^\eta_{\Omega,m}$.
 \end{theorem}
 
\begin{proof}
Using Notation \ref{ueta}, and according to Proposition \ref{P5.1}, let us assume that $\lambda_n\to\lambda_\eta>0$.  
Then,  by \eqref{1356}, we see that $(u_n)_n$ is a bounded Palais-Smale sequence also for the functional
$$
E_{\lambda_\eta,\eta}(u):=\frac{1}{2} \int  | \nabla u |^2 \, dx +\frac{\lambda_\eta}{2} \int_\Omega u^2\, dx- \frac{\eta}{p} \int  | u |^p \, dx \qquad u\in H^1_0(\Omega).
$$
Hence, taking into account Lemma \ref{LperPS}, we can proceed exactly as in \cite[\S 3.3]{MR4304693} and conclude that $(u_n)_n$ strongly converges in $H^1_0(\Omega)$ to the function $u_\eta$.
Then, clearly, $I_\eta(u_\eta)=c^\eta_{\Omega,m}$ and $u_\eta$ is a critical value for the functional $E_{\lambda_\eta,\eta}$, so a constrained critical point for $I_\eta$  constrained on $S_m$.
\end{proof}
     
\medskip


\section{Blow-up analysis and proof of Theorem \ref{T1}}
\label{S6}


In this section, we are working in the following framework.

\begin{notat}
\label{nota2}
Let $(\eta_n)_n$ be a sequence in $A_\Omega$ such that  $\eta_n \nearrow 1$, we consider the family $u_{\eta_n}$ of nonnegative solutions of \eqref{eq U} provided by Theorem \ref{Th.eta}, and the corresponding Lagrange multipliers $\lambda_{\eta_n}$, with $\lambda_{\eta_n}>0$ by Proposition \ref{P5.1}.

To simplify the notation, we simply write $u_n=u_{\eta_n}$ and $\lambda_n=\lambda_{\eta_n}$, and we use the same indices up to a subsequence.
 \end{notat}
 
\begin{remark} 
For every $n\in\N$ the solution $u_n$ is regular, hence strictly positive, and decays exponentially (see, for example, \cite[\S 8]{MR1814364}, \cite[\S 6.3.2]{EvansBook}, or also \cite[Theorem 1.4]{CM10} for exponential decay).
\end{remark}

Here, our aim is to analyze the asymptotic behavior of $(u_n)_n$ and $(\lambda_n)_n$, as $n\to\infty$. 
We can easily state a first estimate:

\begin{remark}
\label{Lsat}
If $x_n\in\R^N$ is a local maximum point of $u_n$, then
\[
\eta_n u_n(x_n) \geq \lambda_n^{\frac{1}{p-2}}.
\]
\end{remark}
Indeed, in the maximum points, the following relation holds:
\[
\lambda_n u_n(x_n) - \eta_n u_n(x_n)^{p-1} = \Delta u_n(x_n) \leq 0.
\]

As a consequence of Remark \ref{Lsat}, we see that if the Lagrange multipliers diverge, then $|u_n|_\infty\to \infty$.
The next lemma, basic to analyze in Proposition \ref{th3} the blow-up phenomenon, states  also, in point $(a)$,  that if $|u_n|_\infty$ is unbounded then 
$$
|u_n|_\infty\sim \lambda_n^{\frac{1}{p-2}}.
$$
 
\begin{lemma}
\label{th2}
Assume that for every $n\in  \N$ there exists $P_n \in \Omega$ such that 
\begin{equation*}
u_n(P_n)\to \infty \mbox{ as } n\to\infty, \mbox{ and }  u_n(P_n) = \max_{B_{\tilde  R_n \tilde{\varepsilon}_n}(P_n)} u_n,
\end{equation*}
for some $\widetilde R_n \to \infty$, where $\tilde{\varepsilon}_n:=u_n(P_n)^{-\frac{p-2}{2}}$, 
and define
\begin{equation*}
U_n(y) = \varepsilon_n^{\frac{2}{p-2}} u_n(\varepsilon_n y + P_n), \quad y \in \Omega_n = \frac{\Omega - P_n}{\varepsilon_n},
\end{equation*}
with $\varepsilon_n = \lambda_n^{-\frac{1}{2}}$. 
Then, for a subsequence, we have: 
\begin{itemize}
    \item[(a)]  as $n \to \infty$, it holds that $\frac{\varepsilon_n} {d(P_n, \partial \Omega)} \to 0$, and $\frac{\tilde \varepsilon_n}{\varepsilon_n}\to [U(0)]^{-\frac{p-2}{2}}\in(0,1]$,  where $U$ is the unique positive solution of \eqref{Pinf};
    \item[(b)]  $u_n(P_n) = \max\limits_{B_{R_n \varepsilon_n}(P_n)} u_n$,  where $R_n=\tfrac{\tilde \e_n}{\e_n}\widetilde R_n$;
    \item[(c)]  $U_n \to U$ in $C^1_{\text{loc}}(\mathbb{R}^N)$, as $n \to \infty$;
    \item[(d)]  there exists $\psi_n \in C_0^\infty(\Omega)$, with $\text{supp}\,\psi_n \subset B_{R \varepsilon_n}(P_n)$   for a suitable $R>0$, such that, for large $n$,
    \begin{equation} \label{eq12}
    \int \big[|\nabla \psi_n|^2 + \left(\lambda_n   - (p-1) u_n^{p-2} \right) \psi_n^2 \big]\, dx < 0;
    \end{equation}
    \item[(e)]  for all $R > 0$, there holds
    \begin{equation*}
    \lim_{n \to +\infty} \lambda_n^{\frac{N}{2}-\frac{2}{p-2}} \int_{B_{R \varepsilon_n}(P_n)} u_n^{2} \, dx = 
    \int_{B_R(0)} U^{2} \, dx.
    \end{equation*}  
\end{itemize}
\end{lemma}

\begin{proof}
    We define $ \widetilde{U}_n(y) = \tilde{\varepsilon}_n^{\frac{2}{p-2}} u_n(\widetilde{\varepsilon}_n y + P_n) $ for $ y \in \tilde{\Omega}_n = \frac{\Omega - P_n}{\tilde{\varepsilon}_n}$ and set $d_n := \text{dist} \left(P_n, \partial \Omega \right)$. 
    We can assume that $\frac{\tilde{\varepsilon}_n}{d_n} \to L \in [0,+ \infty]$, up to a subsequence, for $n \to \infty$. 

    \medskip
    
{\underline {\em We claim that $L=0$}}. \quad     
Assume, by contradiction, that  $L\in (0, +\infty]$.
Then $\Omega_n \to H$, where $H$ denotes a half-space such that $0 \in \overline H$ and $\text{dist}(0, \partial H) = {1}/{L}$. 
The function $\widetilde{U}_n$ satisfies
\begin{equation*}
\begin{cases}
-\Delta \widetilde{U}_n + \lambda_n \tilde{\varepsilon}_n^2  \widetilde{U}_n = \eta_n \widetilde{U}_n^{p-1} & \text{in } \widetilde{\Omega}_n, \\
0 < \widetilde{U}_n \leq \widetilde{U}_n(0) = 1 & \text{in } \widetilde{\Omega}_n \cap B_{R_n}(0), \\
\widetilde{U}_n = 0 & \text{on } \partial \widetilde{\Omega}_n.
\end{cases}
\end{equation*}
Since $P_n$ is a point of local maximum for $u_n$, we deduce that
\begin{equation*}
0 \leq \Delta \widetilde{U}_n (0) = \eta_n - \lambda_n \tilde{\varepsilon}_n^2,
\end{equation*}
which implies that $\lambda_n \tilde{\varepsilon}_n^2 \leq \eta_n \to 1$ as $n \to \infty$. 
Thus, up to a subsequence, $\lambda_n \tilde{\varepsilon}_n^2 \to \tilde{\lambda}$ for some $ \tilde{\lambda} \in [0, 1]$. 
Now, from standard elliptic regularity \cite{MR1814364}, we can conclude that, up to a further subsequence, $\widetilde{U}_n \to \widetilde{U}$ in $C^1_{\text{loc}}(\bar{H})$, where $\widetilde{U}$ is a nontrivial solution of
\begin{equation}
\label{EL}
\begin{cases}
-\Delta \widetilde{U} + \tilde{\lambda} \widetilde{U} = \widetilde{U}^{p-1} & \text{in } H, \\
0 < \widetilde{U} \leq \widetilde{U}(0) = 1 & \text{in } H, \\
\widetilde{U} = 0 & \text{on } \partial H.
\end{cases}
\end{equation}
Observe that \eqref{EL} has no solution (see \cite[Theorem 1.3]{MR619749} for $\tilde\lambda=0$ and \cite[Theorem 1.1]{MR2825606} for $\tilde\lambda>0$).
Then we have a contradiction, and the claim follows.

\smallskip
 
Hence, we have to analyze the case $\frac{\tilde{\varepsilon}_n}{d_n} \to 0$, that is $H = \mathbb{R}^N$ in \eqref{EL}.
Furthermore, notice that by \cite[Theorem 1.1]{GS81CPAM} problem \eqref{EL} has only the trivial solution when $H=\R^N$ and $\tilde\lambda=0$, so we are left to consider when $\tilde\lambda>0$. 
 At this point, we also get the following
\begin{equation} \label{eq23}
 \frac{\tilde{\varepsilon}_n}{\varepsilon_n}  = {\lambda}_n^{1/2} \tilde{\varepsilon}_n  \to \tilde{\lambda}^{1/2} \in (0, 1] \quad \text{as } n \to \infty.
\end{equation}
Moreover, it is more convenient to work with the function $U_n$, that satisfies
\begin{equation*}
\begin{cases}
-\Delta U_n + U_n = \eta_n U_n^{p-1} & \text{in } \Omega_n, \\
0 < U_n \leq U_n(0) = \left( \frac{\tilde{\varepsilon}_n}{\varepsilon_n} \right)^{-\frac{2}{p-2}} & \text{in } \Omega_n \cap B_{R_n \frac{\tilde{\varepsilon}_n}{\varepsilon_n}}(0), \\
U_n = 0 & \text{on } \partial \Omega_n.
\end{cases}
\end{equation*}
Then, $U_n \to U$ in $C^1_{\text{loc}}(\R^N)$ as $n \to \infty$, where $U$ is a non-trivial bounded solution of 
\begin{equation}
\label{Kwo}
\begin{cases}
-\Delta U +  U = U^{p-1} & \text{in } \R^N, \\
0 < U \leq U(0)  = (\tilde\lambda)^{-\frac{1}{p-2}} & \text{in } \R^N.  
\end{cases}
\end{equation}
Now, we are proving that
\begin{equation}
\label{1733}
i(U) \leq \sup_n i(u_n) \le N+2,
\end{equation}
where the last inequality comes from Proposition \ref{Pmorse}.
If $i(U)\ge k\in\N$, then one can find  $\varphi_1, \dots, \varphi_k \in C_0^\infty(\mathbb{R}^N)$  
  orthogonal in $L^2(\R^N)$ such that
\begin{equation*}
\int_{\mathbb{R}^N} \big[|\nabla \varphi_i|^2 + (1 - (p-1) {U}^{p-2}) \varphi_i^2\big]\, dx < 0 \quad \forall i = 1, \dots, k.
\end{equation*}
It is easily seen that the functions $\varphi_{i,n}(x) := {\varepsilon}_n^{-\frac{N-2}{2}} \varphi_i\left(\frac{x - P_n}{{\varepsilon}_n}\right)$  are supported in $\Omega$ (because $\frac{{\varepsilon}_n}{d_n}=\frac{{\varepsilon}_n}{\tilde{\varepsilon}_n}\frac{\tilde{\varepsilon}_n}{d_n} \to 0$), are orthogonal in $L^2(\Omega)$ and satisfy
\begin{align*}
\int_{\Omega} \big[|\nabla \varphi_{i,n}|^2 + (\lambda_n - \eta_n (p-1) u_n^{p-2}) \varphi_{i,n}^2\big]\, dx
= & \int_{\Omega_n} \big[|\nabla \varphi_i|^2 + (1 - \eta_n(p-1)  {U}_n^{p-2}) \varphi_i^2\big]\, dx \\
& \to \int_{\mathbb{R}^N} \big[|\nabla \varphi_i|^2 + (1 - (p-1)  {U}^{p-2}) \varphi_i^2\big]\, dx < 0,
\end{align*}
as $n \to \infty$, for all $i = 1, \dots, k$. 
Hence,
\begin{equation*}
   \int_{\Omega} |\nabla \varphi_{i,n}|^2 + (\lambda_n - \eta_n(p-1) u_n^{p-2}) \varphi_{i,n}^2 dx < 0
\end{equation*}
for sufficiently large $n$ and for every $i=1,\cdots, k$.
So,  $k$ has to be less than or equal to $N+2$, and \eqref{1733} is proved.

By \cite[Theorem 1.1]{MR2825606}, we conclude that $U$ coincides with the solution of \eqref{Pinf}.
Therefore, the point $(a)$ is fully demonstrated, in light of \eqref{eq23} and \eqref{Kwo}.

Moreover, since $U$ is an unstable solution,  there exists $\psi \in C_0^\infty(\mathbb{R}^N)$ such that $\operatorname{supp} \psi \subset B_R(0)$, $R > 0$, and
\begin{equation*}
\int_{\mathbb{R}^N} \big[|\nabla \psi|^2 + (1 - (p-1) U^{p-2}) \psi^2 \big]\, dx< 0.
\end{equation*}
Then, the function $\psi_n(x) = \e_n^{-\frac{N-2}{2}} \psi\left(\frac{x - P_n}{\e_n}\right)$ satisfies \eqref{eq12}, for large $n$.

To conclude the proof, we observe that for every $R>0$
\begin{equation*}
\int_{B_R(0)} U^{2} = \lim_{n \to \infty} \int_{B_R(0)} U_n^{2}
= \lim_{n \to +\infty} \lambda_n^{\frac{N}{2}-\frac{2}{p-2}} \int_{B_{R \varepsilon_n}(P_n)} u_n^{2} \, dx ,
\end{equation*}
proving also point $(e)$.
\end{proof}

\begin{proposition}
 \label{th3} 
Let $(u_n)_n$ and $(\lambda_n)_n$ be as in Notation \ref{nota2}.
If $\lambda_n\to\infty$, then, up to a subsequence, there exist at most $k$ sequences of points $(P_n^1)_n, \dots, (P_n^k)_n$, with $k \leq N+2$, such that for all $i, j = 1, \dots, k$, $i \neq j$, we have
\begin{equation} \label{eq13}
\lambda_n |P_n^i - P_n^j|^2 \to \infty \quad \text{and} \quad \lambda_n d(P_n^i, \partial \Omega)^2 \to \infty, \quad \text{as } n \to \infty,
\end{equation}
\begin{equation} \label{eq14}
u_n(P_n^i) = \max_{ B_{\e_n R_n }(P_n^i)} u_n,\ \mbox{ for some }R_n \to \infty,\ \mbox{ as }n \to \infty,
\end{equation}
 where  $\e_n=\lambda_n^{-1/2}$, and the functions
\begin{equation*}
U_n(y) = \varepsilon_n^{\frac{2}{p-2}} u_n(\varepsilon_n y + P_n^i), \quad y \in \Omega_n = \frac{\Omega - P_n^1}{\varepsilon_n},\quad i-1,\ldots,k,
\end{equation*}
satisfy $(c)-(e)$ in Lemma \ref{th2}.
Moreover, there holds 
\begin{equation} \label{eq10}
u_n(x) \leq C \lambda_n^{\frac{1}{p-2}} \sum_{i=1}^k e^{-\gamma  \lambda_n^{\frac{1}{2}} |x - P_n^i|} \quad \forall x \in \Omega, \, n \in \mathbb{N},
\end{equation}
for some constants $C, \gamma > 0$.
\end{proposition}

\begin{proof}
   We divide the proof into two main steps. We argue up to suitable subsequences.

\noindent \textbf{Step 1.} There exist $k \leq {N+2}$ sequences $P_n^1, \ldots, P_n^k$ that satisfy the conditions \eqref{eq13} and \eqref{eq14}, and in addition we have
\begin{equation} \label{eq15}
\lim_{R \to \infty} 
\left(
\limsup_{n \to \infty}  
\left[ \e_n^{
\frac{2}{p-2}} 
\max_{d_n(x) \geq R \e_n} 
 u_n(x) 
\right]
\right) = 0, \quad 
\end{equation}
where the function 
\[
d_n(x) = \min_{i=1, \ldots, k} | x - P_n^i |
\]
measures the Euclidean distance from the points $\{P_n^1, \ldots, P_n^k\}$. 

Let us first consider $P_n^1$, which corresponds to the  global maximum of $u_n$, i.e., $u_n(P_n^1) = \max_{\Omega} u_n(x)$. If \eqref{eq15} holds for $\{P_n^1\}$, we can set $k = 1$, and \eqref{eq13}, \eqref{eq14} directly follow, using Lemma \ref{th2}. 

Now, assume that \eqref{eq15} does not hold, for $d_n(x)=|x-P^1_n|$.
Then, we have
\[
\lim_{R \to +\infty} 
\left(
\limsup_{n \to +\infty} 
\left[
\varepsilon_n^{\frac{2}{p-2}} 
\max_{| x - P_n^1| \geq R \varepsilon_n} u_n
\right]
\right) = 4 \delta > 0
\]
for some $\delta>0$. 
 
Invoking Lemma \ref{th2} and Remark \ref{Lsat}, we deduce that
\begin{equation} \label{eq16}
\varepsilon_n^{\frac{2}{p-2}} u_n(\varepsilon_n y + P_n^1) = : U_n^1(y) \to U(y) \quad \text{in } C^1_{\text{loc}}(\mathbb{R}^N), 
\end{equation}
as $n \to \infty$, where $U$ is a solution of \eqref{Pinf}. 
Since $U(y) \to 0$ as $| y | \to \infty$, we can choose $R$ sufficiently large such that
\begin{equation} \label{eq17}
U(y) \leq \delta \quad \text{for all } |y| \geq R. 
\end{equation}
By further increasing $R$ if necessary, we can assume without loss of generality that
\begin{equation} \label{eq18}
\varepsilon_n^{\frac{2}{p-2}} 
\max_{| x - P_n^1| \geq R \varepsilon_n} u_n \geq 2\delta.
\end{equation}
From $u_n = 0$ on $\partial \Omega$ and $u_n(x) \to 0$ as $|x| \to \infty$, it follows that there exists a point $P_n^2 \in   B^c_{R \varepsilon_n}(P_n^1)$ such that
\beq
\label{1707}
u_n(P_n^2) = \max_{  B^c_{R \varepsilon_n}(P_n^1)} u_n.
\eeq
Coupling \eqref{eq16} and \eqref{eq17}, we can see that
\beq
\label{1654}
\frac{| P_n^2 - P_n^1 |}{\varepsilon_n} \to \infty.
\eeq
Indeed, if   $\frac{| P_n^2 - P_n^1 |}{\varepsilon_n} \to R' \geq R$, then we would have
\[
(\varepsilon_n)^{\frac{2}{p-2}} u_n(P_n^2) = U_n^1\left(\frac{P_n^2 - P_n^1}{\varepsilon_n}\right) \to U(R') \leq \delta,
\]
contradicting \eqref{eq18}. 
This establishes that the first condition in \eqref{eq13} is satisfied for $\{P_n^1, P_n^2\}$. 
Observe that \eqref{eq18} implies $u_n(P^2_n)\to \infty$, as $n\to\infty$.
Then, define
\[
\tilde{\varepsilon}_{2,n} : = u_n(P_n^2)^{-\frac{p-2}{2}},
\]
and let
\beq\label{1701}
\widetilde R_{2,n} = \frac{1}{2} \frac{|P_n^2 - P_n^1 |}{\tilde{\varepsilon}_{2,n}}.
\eeq
From \eqref{eq18}, it follows that
\[
\tilde{\varepsilon}_{2,n} \leq (2\delta)^{-\frac{p-2}{2}} \varepsilon_n,
\]
which implies
\[
\widetilde R_{2,n} \geq \frac{(2\delta)^{\frac{p-1}{2}}}{2} \frac{| P_n^2 - P_n^1 |}{\varepsilon_n} \to \infty \quad \text{as } n \to \infty.
\]
 
This ensures that
\beq
\label{1706}
u_n(P_n^2) = \max_{  B_{\tilde R_{2,n} \tilde{\varepsilon}_{2,n}}(P_n^2)} u_n.
\eeq

Indeed, from \eqref{1701} and \eqref{1654} we infer
\[
| x - P_n^1 | \geq | P_n^2 - P_n^1 | - | x - P_n^2 | \geq \frac{1}{2} | P_n^2 - P_n^1| \geq R \varepsilon_n,\qquad \forall x \in B_{\tilde R_{2,n} \tilde{\varepsilon}_{2,n}}(P_n^2).
\]
Thus
\[
 B_{\tilde R_{2,n} \tilde{\varepsilon}_{2,n}}(P_n^2) \subset B^c_{R \varepsilon_n}(P_n^1)
\]
and \eqref{1706} follows from \eqref{1707}.

Since $\widetilde R_{2,n} \to \infty$ as $n \to \infty$, Lemma \ref{th2} guarantees that the second condition in \eqref{eq13} and \eqref{eq14} are also satisfied, for $\{P_n^1, P_n^2\}$, by choosing $R_{n}:=\tfrac{\tilde\e_{2,n}}{\e_n}\widetilde R_{2,n}$ (we can also observe that $R_{n}=\tfrac12 |P^2_n-P^1_n|$). 
If \eqref{eq15} holds for $ \{P_n^1, P_n^2\} $, the proof of Step 1 is concluded. 

\smallskip
 
Otherwise, the procedure is continued iteratively. Suppose there exist $P_n^1, \ldots, P_n^s$ such that \eqref{eq13} and \eqref{eq14} are satisfied but \eqref{eq15} fails. As previously, we select $R > 0$ sufficiently large and extract a subsequence such that
\[
\left(\varepsilon_n\right)^{\frac{2}{p-2}} \max_{\{d_n(x) \geq R \varepsilon_n\}} u_n(x) \geq 2 \delta,
\]
where
\[
d_n(x) = \min_{i=1, \ldots, s} | x - P_n^i |.
\]
Using Lemma \ref{th2}, we infer that for every $i\in\{1,\ldots,s\}$
\begin{equation} \label{eq19}
\left(\varepsilon_n\right)^{\frac{2}{p-2}} u_n\left(\varepsilon_n y + P_n^i\right) = :U_n^i\left( y\right) \to  U( y),
\end{equation}
 in $C^1_{\text{loc}}(\mathbb{R}^N)$ as $n \to \infty$.  
 Now, we repeat the above argument, starting by defining $P_n^{s+1}$ such that
\begin{equation} \label{eq20}
 u_n(P_n^{s+1}) = \max_{\{d_n(x) \geq R \varepsilon_n\}} u_n(x) \geq 2 \delta\varepsilon_n^{-\frac{2}{p-2}}. \end{equation}
From \eqref{eq19} and the condition $ U(y) \leq \delta$ for $|y| \geq R$, it follows as before that
\[
\frac{| P_n^{s+1} - P_n^i |}{\varepsilon_n} \to \infty, \quad \text{as } n \to \infty, \quad \text{for all } i = 1, \ldots, s.
\]
Thus, the first condition in \eqref{eq13} holds for $\{P_n^1, \ldots, P_n^{s+1}\}$. 
Setting $\tilde{\varepsilon}_{s+1,n} = u_n(P_n^{s+1})^{-\frac{p-2}{2}}$ and $\widetilde  R_{s+1,n} = \frac{1}{2} \frac{d_n(P_n^{s+1})}{\tilde{\varepsilon}_{s+1,n}}$, \eqref{eq20} implies that
\[
\tilde{\varepsilon}_{s+1,n} \leq (2 \delta)^{-\frac{p-2}{2}} \varepsilon_n,  
\]
and, as a consequence, that $\widetilde R_{s+1,n} \to \infty$  as  $n \to \infty$. 
Moreover, arguing as above, we get 
\[
u_n(P_n^{s+1}) = \max_{  B_{R_{s+1,n} \tilde{\varepsilon}_{s+1,n}}(P_n^{s+1})} u_n(x),
\]
and Lemma \ref{th2} ensures that the second condition in \eqref{eq13}, and \eqref{eq14}, hold also for $\{P_n^1, \ldots,$ $P_n^{s+1}\}$. 
Furthermore, Lemma \ref{th2} guarantees the existence of functions $\{\psi_n^i\}_{i=1}^{s+1} \subset C_c^0(\Omega)$ satisfying \eqref{eq12}, with $\text{supp}(\psi_n^i) \subset B_{R \e_n} (P_n^i)$, for a suitable $R>0$.
From \eqref{eq13}, we infer that the supports of $\psi_n^1, \ldots, \psi_n^{s+1}$ are disjoint for large $n$.
As a consequence, 
\[
s+1 \leq \limsup_{n \to +\infty} i(u_n)\le N+2.
\]
So, the iterative process stops after $k$ steps, with $k \leq N+2$, and the existence of sequences $\{P_n^1, \ldots, P_n^k\}$ that satisfy conditions \eqref{eq13}, \eqref{eq14}, and \eqref{eq15} is established.

\medskip

\noindent \textbf{Step 2.} Let $P_n^1, \dots, P_n^k$ be defined as in Step 1. Then, there exist constants $\gamma$ and $C > 0$ such that  
\begin{equation*}
    u_n(x) \leq C \lambda_n^{\frac{1}{p-2}} \sum_{i=1}^k e^{-\gamma \lambda_n^{\frac{1}{2}} |x - P_n^i|} \quad \forall x \in \Omega, \quad n \in \mathbb{N}.
\end{equation*}

Using \eqref{eq15}, for a sufficiently large $R > 0$ and $n \geq n(R)$, we deduce that  
\begin{equation} \label{eq21}
    \lambda_n^{-\frac{1}{p-2}}\max_{d_n(x) \geq R \lambda_n^{-1/2}} u_n(x) \leq  {2^{-\frac{1}{p-2}}}.
\end{equation}  
 
Thus, for large $n$,
\begin{equation} 
\label{eq24}
    \tilde{a}_n(x): =\lambda_n - \eta_n u_n^{p-2}(x) \geq \lambda_n \left(1-\frac{\eta_n}{2} \right) \geq \frac{\lambda_n }{4}, \quad \text{in } \{d_n(x) > R\lambda^{-1/2}_n\}. 
\end{equation}  
Next, let us calculate the  the linear operator $-\Delta + \tilde{a}_n(x)$ on the function  
\begin{equation*}
    \psi^i_n(x) := e^{-\gamma \lambda_n^{\frac{1}{2}} |x - P^i_n|}
\end{equation*}
within the region $\{d_n(x) \geq R\lambda_n^{-1/2}\}$. For sufficiently large $n$, thanks to \eqref{eq24} we find that  
\begin{equation} \label{eq22}
    \left(  -\Delta + \tilde{a}_n(x) \right) \psi_i^n = \lambda_n \psi^i_n \left[ -\gamma^2 + (N-1)\frac{\gamma}{\lambda_n^{\frac{1}{2}} |x - P^i_n|} + \lambda_n^{-1}\tilde{a}_n(x) \right] > 0,
\end{equation}  
provided that $\gamma$ is chosen small enough. 
Moreover, for large $R$, we observe that  
\begin{equation}
\label{eq23b}
   \left( e^{\gamma R } \psi_n^i(x)- \lambda_n^{-\frac{1}{p-2}} u_n(x)\right)_{\big| \partial B_{R \lambda_n^{- {1}/{2}}}(P_n^i)} \to 1-  U( R) > 0,\ \mbox{ as }n \to \infty.
\end{equation}
 Define  
\begin{equation*}
    \psi_n = e^{\gamma R} \lambda_n^{\frac{1}{p-2}} \sum_{i=1}^k \psi^i_n,
\end{equation*}
and consider the operator $\mathcal{L}_n := -\Delta+\lambda_n-\eta_n u_n^{p-2}$. Note that $\mathcal{L}_n u_n = 0$. In view of \eqref{eq22}, we deduce  
\begin{equation*}
    \mathcal{L}_n (\psi_n - u_n) =  e^{\gamma R} \sum_{i=1}^k \left(  -\Delta + \tilde{a}_n(x) \right) \psi_i^n \geq 0 \quad \text{in } \{d_n(x) > R\lambda^{-1/2}_n\}.
\end{equation*}
Furthermore, by using \eqref{eq23b} we get $\psi_n - u_n \geq 0$ on $\{d_n(x) = R\lambda^{-1/2}_n\} \cup \partial \Omega$, while $(\psi_n - u_n)(x) \to 0$ as $|x| \to \infty$. Additionally, from \eqref{eq13} and \eqref{eq14}, we see that  
\begin{equation*}
\{d_n(x) = R\lambda^{-1/2}_n\} = \cup_{i=1}^k \partial B_{R \lambda_n^{-\frac{1}{2}}}(P_n^i) \subset \Omega.
\end{equation*}
Thus, by the minimum principle,  
\begin{equation*}
    u_n \leq \psi_n = e^{\gamma R} \lambda_n^{\frac{1}{p-2} }\sum_{i=1}^k e^{-\lambda_n^{\frac{1}{2}} |x - P^i_n|}
\end{equation*}
within $\{d_n(x) \geq R \lambda^{-1/2}_n\}$, for large $R$ and $n \geq n(R)$. 
Finally, using Lemma \ref{th2} $(a)$, we conclude that  
\begin{equation*}
    u_n(x) \leq \max_{\Omega} u_n = \tilde{\varepsilon}_n^{-\frac{2}{p-2}} \leq C \lambda_n^{\frac{1}{p-2}}e^{\gamma R} \sum_{i=1}^k e^{-\lambda_n^{\frac{1}{2}} |x - P_i^n| },
\end{equation*}
for some $C > 0$, when $d_n(x) \leq R\lambda^{-1/2}_n$. 
Therefore, \eqref{eq10} holds throughout $\Omega$ with a constant $C e^{\gamma R}$ for all $n \geq n(R)$. 
Enlarging $C$ if necessary, we see that \eqref{eq10} is valid in $\Omega$ for every $n \in \mathbb{N}$.

\end{proof}
 
{\em Proof of Theorem \ref{T1}}.\  Here, we consider the sequences $(\lambda_n)_n \subset \R$ and $(u_n)_n \subset H^1_0(\Omega)$ as in Notation \ref{nota2}.

\smallskip

\noindent \textbf{Step 1.}  {\em $(\lambda_n)_n \subset \R$ is bounded.}

\smallskip

    Arguing by contradiction, let as assume that $\lambda_n \to  \infty$. 
Using the notation of Proposition \ref{th3}, we claim 
    \begin{equation} \label{eq11}
        \lim_{n \to \infty}\left| \lambda_n^{\frac{N}{2}-\frac{2}{p-2}} \int_\Omega u_n^2 \, dx - \sum_{j=1}^k \int_{B_R(0)} \left( \widetilde{U}_n^j \right)^2 \, dx  \right|=\infty
    \end{equation}
    where $\widetilde{U}_n^j (x) = \varepsilon_n^{\frac{2}{p-2}} u_n(P_n^j+ \varepsilon_n x)$, with $\varepsilon_n = \lambda_n^{-\frac{1}{2}}$, and $R>0$. 
    In fact, on the one hand 
    \begin{equation*}
        \lim_{n \to \infty} \lambda_n^{\frac{N}{2}-\frac{2}{p-2}} \int_\Omega u_n^2 \, dx= \lim_{n \to \infty} \lambda_n^{\frac{N}{2}-\frac{2}{p-2}} m^2= \infty,
    \end{equation*}
    being $\frac{N}{2}-\frac{2}{p-2}>0$,  and on the other hand
    \begin{equation*}
        \lim_{n \to \infty }\int_{B_R(0)} \left(\widetilde{U}_n^j\right)^2 \, dx = \lim_{n \to \infty} \lambda_n^{\frac{N}{2}-\frac{2}{p-2}} \int_{B_R \varepsilon_n(P_n^j)} u_n^2 \, dx =  \int_{B_R(0)} U_0^2 \, dx < \infty,\ i=1,\ldots,k,
    \end{equation*}
   by Lemma \ref{th2} $(e)$. 
   Since $k\le N+2$, the claim follows. 

  \smallskip
  
   Now, from \eqref{eq10} we infer
   \begin{align*}
        \lim_{n \to \infty} &\left| \lambda_n^{\frac{N}{2}-\frac{2}{p-2}} \int_\Omega u_n^2 \, dx - \sum_{j=1}^k \int_{B_R(0)} \left( \widetilde{U}_n^j \right)^2 \, dx  \right|=\lim_{n \to \infty} \lambda_n^{\frac{N}{2}-\frac{2}{p-2}} \int_{\left(\cup_{j=1}^k B_{R \varepsilon_n}(P_n^j)\right)^c} u_n^2 \, dx \\
        & \quad \leq C \lambda_n^{\frac{N}{2}}\int_{\R^N} \sum_{i=1}^k e^{-2 \gamma  \lambda_n^{\frac{1}{2}} |x - P_n^i|} \leq \widetilde{C}\int_0^{+\infty} \varrho^{N-1} e^{-\varrho}\, d \varrho < \infty
    \end{align*}
    for some $C, \widetilde{C}>0$. Comparing this with \eqref{eq11} we have a contradiction. 

\medskip

\noindent \textbf{Step 2.} {\em $(u_n)_n \subset H^1_0(\Omega)$ is bounded.} 

\smallskip

Since
\[
\int  \left(|\nabla u_n|^2 + \lambda_n u_{n}^2\right)\,dx = \eta_n \int  |u_n|^p\, dx,
\]
we obtain
\[
c^{\eta_n}_{\Omega,m} = I_{\eta_n}(u_{n}) = \left( \frac{1}{2} - \frac{1}{p}\right) \int_{\Omega} |\nabla u_{n}|^2\,dx - \frac{\lambda_n m}{p}.
\]
Thus,
\begin{equation}\label{eq27}
\left( \frac12 -\frac{1}{p}\right) \int_{\Omega} |\nabla u_{n}|^2\,dx = c_{\Omega,m}^{\eta_n} + \frac{\lambda_n m}p.
\end{equation}
From this relation, the boundedness of $(u_n)_n$ in $H^1_0(\Omega)$ follows, by Lemma \ref{LperPS} and because $(\lambda_n)_n$ is bounded. 

\medskip

\noindent \textbf{Step 3.} {\em Fixed $\varrho_m=\varrho_{ps}$,  if $\Omega^c\subset B_{\varrho_m}$ then $u_n\to u$ in $H^1_0(\Omega)$ for some $u\in S_m$, up to a subsequence.}

\smallskip

First, observe that $(u_n)_n$ bounded in $H^1(\Omega)$ and $\eta_n \to 1$, imply that $(u_n)_n$ is a Palais-Smale sequence for $I$ constrained on $S_m$, at the level $c_{\Omega,m}$.
 Moreover, since $(\lambda_n)_n$ is also bounded, and $\lambda_n>0$, we can also assume that $\lambda_n\to\lambda\in[0,\infty)$. 

If $u_n \rightharpoonup u$ weakly but not strongly in $H^1_0(\Omega)$, up to a subsequence, then we can argue exactly as in Proposition \ref{P5.1}, and conclude that either $\lambda =0$ and $|u|_2=m$ thanks to Hypothesis \eqref{NEm} or $\lambda>0$. 
 Hence, taking into account Lemma \ref{LperPS}, the proof in \cite[\S 3.3]{MR4304693} works, proving that, in fact, $u_n\to u$ strongly in $H^1_0(\Omega)$.
 
\smallskip

\noindent \textbf{Step 4.}  {\em $u$ is the solution we are looking for.}

\smallskip

Since $(u_n)_n$ is a PS sequence, $u_n\to u$ strongly in $H^1_0(\Omega)$, and $u_n\ge 0$, then $u$ is a nonnegative solution of \eqref{Pm}.
Then, we can conclude that $u>0$ by the Harnack inequality.

For the estimate on the Morse index, one can proceed as in Proposition \ref{Pmorse}.
 
 \qed
 

\section{Proof of Theorem \ref{T2}}


To prove Theorem \ref{T2}, we use the same arguments as used in the proof of Theorem \ref{T1}.
If $\Omega^c\subset B_{\bar\varrho}$, with $\bar\varrho$ fixed arbitrarily large, we will consider the asymptotic behavior of the energy of the test functions $Z_{m,\bar\varrho}[y,h]$, as $m\to\infty$, and verify that for large $m$ there is a linking structure as in \S \ref{S3}, in the compactness range highlighted in Lemma \ref{LperPS}.

\smallskip
\begin{lemma}
\label{L7}
Let $\Omega$ an open set in $\R^N$, if $\Omega^c\subseteq B_{\bar\varrho}$, then there exist $m_{\Omega}$, and $\bar h >0$, such that, 
setting $Z_m:=Z_{m,\bar\varrho}$ as in \eqref{1624}, for every $m>m_{\Omega}$,
\beq
\label{1723}
  \max_{(y,h)\in  S}I(Z_m[y,h])<\min\left\{2,1+\frac2N\right\}c_m,
\eeq
where $S:=\R^N\times[-\bar h,\bar h]$ and $c_{\Omega,m}$ has been introduced in \eqref{1453}.
\end{lemma}
\proof
In the notation of \S \ref{S3}, from direct computations we infer
\beq
\label{1732}
I( \widetilde Z_m[y,h])=\mu^{-\left(\frac{2p}{p-2}-N\right)}I(h\ast U ),\qquad \forall (y,h)\in\R^N\times \R,
\eeq
where $U=\widetilde Z_{m_0}[0,0]$.
Here, $\mu=\mu(1,m)$ verifies $\mu\to\infty$ as $m\to\infty$ (see \eqref{1725}).
From \eqref{1732} it follows, in particular,  that $c_m=I(\widetilde Z_m[0,0])=\mu^{-\left(\frac{2p}{p-2}-N\right)}c_{m_0}$,
because $c_{m_0}=I(U)=\max_{h\in\R} I(h\ast  U )$.
Then, taking into account that $U$ is a mountain pass solution, and the behavior of the function $h\mapsto I(h\ast U)$, it is easily seen that $\bar h>0$ exists such that 
\beq\label{1718}
I(\widetilde Z_m[y,\bar h])<0,\qquad I(\widetilde Z_m[y,-\bar h])<\frac12 c_m,\qquad\forall m>0,\ \forall y\in\R^N.
\eeq
Since, by standard computations, we can see
$$
Z_m[y,\pm \bar h] \to \widetilde Z_m[y,\pm \bar h]\ \mbox{ in }H^1(\R^N),\mbox{ as } m\to\infty, \mbox { uniformly in }y\in\R^N,
$$
we get that for large $m$ inequalities in \eqref{1718} hold with $Z_m$ in place of $\widetilde Z_m$. 

Our next claim is to show
\beq
\label{eas}
I(Z_m[y,h])=\mu^{-\left(\frac{2p}{p-2}-N\right)}I(h\ast U )+o(\mu^{-\left(\frac{2p}{p-2}-N\right)}),\ \mbox{ uniformly for }(y,h)\in \R^N\times[-\bar h,\bar h].
\eeq
By \eqref{1624} and \eqref{1727}, to prove \eqref{eas} we have to estimate 
\beq
\label{(iii)}
\begin{array}{cl}
(i)   &\qquad \displaystyle{\int |\vartheta(x)\widetilde Z_m [y,h] |^pdx}\\[10pt]
(ii)  &\qquad \displaystyle{\int |\D \big(\vartheta(x)\widetilde Z_m [y,h] \big)|^2dx}\\[10pt]
(iii) &\qquad \displaystyle{\int |\vartheta(x)\widetilde Z_m [y,h] |^2dx},\\
\end{array}
\eeq
where in the cut-off function the fixed radius $\bar\varrho$ has been omitted.

As for $(i)$:
\begin{equation}
    \label{i}
\begin{split}
\int |\vartheta\widetilde Z_m [y,h] |^pdx=&\int |\widetilde Z_m [y,h] |^pdx+\int_{B_{2\bar\varrho}} (\vartheta^p-1)|\widetilde Z_m [y,h] |^pdx\\[4pt]
=& \mu^{-\left(\frac{2p}{p-2}-N\right)}\int |h\ast U|^pdx+o\left(\mu^{-\left(\frac{2p}{p-2}-N\right)} \right) ;
\end{split}
\end{equation}
uniformly for $(y,h)\in \R^N\times[-\bar h,\bar h]$.
Indeed
$$
\int |\widetilde Z_m [y,h] |^pdx=(e^h)^{\frac{N}{2}(p-2)}\int |\widetilde Z_m[0,0]|^pdx=\mu^{-\left(\frac{2p}{p-2}-N\right)} \int |h\ast U|^pdx,
$$
and
\begin{equation*}
 \begin{split}
\left|\int_{B_{2\bar\varrho}} (\vartheta^p-1)|\widetilde Z_m [y,h] |^pdx\right|
&\le 
(e^{\bar h})^{\frac{N}{2}p}\int_{B_{2\bar\varrho} } |\widetilde Z_m (e^h(x-y))|^pdx\\[4pt]
&\le
c\, \mu^{-\frac{2p}{p-2}}\int_{ B_{2\bar\varrho}}\max_{\R^N} U^pdx\\[4pt]
 &=
 o\left(\mu^{-\left(\frac{2p}{p-2}-N\right)} \right).
 \end{split}
\end{equation*}
 As for $(ii)$, we are proving:
\begin{equation}
    \label{ii}
\begin{array}{rlcc}
\displaystyle{\int |\D \big(\vartheta(x)\widetilde Z_m [y,h] \big)|^2dx} =&\displaystyle{\int \vartheta^2|\D \widetilde Z_m [y,h]|^2dx} & &  (a) \\[6pt]
&+\displaystyle{2\int_{B_{2\bar\varrho} } \vartheta\D\vartheta \cdot \widetilde Z_m [y,h]  \D \widetilde Z_m [y,h] \, dx} & &  (b)\\[6pt]
& +\displaystyle{ \int_{B_{2\bar\varrho} }|\D  \vartheta(x)|^2 |\widetilde Z_m [y,h] |^2dx} & &   (c) \\[6pt]
 = & \mu^{-\left(\frac{2p}{p-2}-N\right)} \int |\D(h\ast U)|^2dx+o\left( \mu^{-\left(\frac{2p}{p-2}-N\right)}\right),
\end{array}
\end{equation}
uniformly for $(y,h)\in \R^N\times[-\bar h,\bar h]$.
Indeed, let us estimate $(a), (b)$, and $(c)$:
\begin{equation}
\tag{$a$}
\begin{split}
\int \vartheta^2|\D \widetilde Z_m [y,h]|^2dx &=   \int  |\D \widetilde Z_m [y,h]|^2dx+\int_{B_{2\bar\varrho} } (\vartheta^2-1)|\D \widetilde Z_m [y,h]|^2dx  \\
&= \mu^{-\left(\frac{2p}{p-2}-N\right)} \int|\D (h\ast U)|dx+ o\left( \mu^{-\left(\frac{2p}{p-2}-N\right)} \right),
\end{split}
\end{equation}
because of the same computations as in $(i)$ and
$$
\left|\int_{B_{2\bar\varrho} } (\vartheta^2-1)|\D \widetilde Z_m [y,h]|^2dx \right|\le c \mu^{-\frac{4}{p-2}}\int_{B_{2\bar\varrho} } 
(\mu^{-2}\max_{\R^N}|\D U|^2)dx=  o\left( \mu^{-\left(\frac{2p}{p-2}-N\right)} \right),
$$
uniformly for $(y,h)\in \R^N\times[-\bar h,\bar h]$.
Now, since $N>1$, 
\begin{equation}
\tag{$b$}
\begin{split}
\left|2\int_{B_{2\bar\varrho} } \vartheta\D\vartheta \cdot \widetilde Z_m [y,h]  \D \widetilde Z_m [y,h] \, dx\right| &\le
\int_{B_{2\bar\varrho} } |\D \vartheta^2|\widetilde Z_m [y,h]  |\D \widetilde Z_m [y,h]| \, dx\\
&\le c \mu^{-\frac{4}{p-2}}\max_{\R^N}  U\,\left(\mu^{-1}\max_{\R^N} |\D U|\right) |B_{2\bar\varrho}|
\\
&= o\left( \mu^{-\left(\frac{2p}{p-2}-N\right)} \right).
\end{split}
\end{equation}
Finally, taking into account $N\ge 3$,
\begin{equation}
\tag{$c$}
\begin{split} 
\int_{B_{2\bar\varrho} }|\D  \vartheta(x)|^2 |\widetilde Z_m [y,h] |^2dx& \le c\int_{B_{2\bar\varrho}} |\widetilde Z_m [y,h] |^2dx \\
& \le c\,\mu^{-\frac{4}{p-2}}\max_{\R^N}U\, |B_{2\bar\varrho}|\\
& = o\left( \mu^{-\left(\frac{2p}{p-2}-N\right)} \right),
\end{split}
\end{equation}
uniformly for $(y,h)\in \R^N\times[-\bar h,\bar h]$.

As for $(iii)$, arguing as in $(i)$ and $(ii)$:
$$
 \int |\vartheta(x)\widetilde Z_m [y,h] |^2dx=\int | \widetilde Z_m [y,h] |^2dx+\int_{B_{2\bar\varrho}}  (\vartheta^2-1) \widetilde Z_m [y,h] |^2dx=m^2+o(1),
$$
uniformly for $(y,h)\in \R^N\times[-\bar h,\bar h]$.

Hence, \eqref{(iii)}, and \eqref{eas}, are proved. 
As a consequence, 
$$
\sup_{(y,h)\in \R^N\times[-\bar h,\bar h]}I(Z_m[y,h]) \le \sup_{h\in [-\bar h,\bar h]} \left[
\mu^{-\left(\frac{2p}{p-2}-N\right)}I(h\ast U )+o(\mu^{-\left(\frac{2p}{p-2}-N\right)})\right]=c_m(1+o(1)),
$$
and so \eqref{1723} holds.
\qed

\medskip

{\em Proof of Theorem \ref{T2}}.\quad
In the notation of Lemma \ref{L7}, we proceed to show that for every $m>m_\Omega$ there exists $\bar R$ such that the linking structure described in \S \ref{S3} arises, setting $Q:=B_{\bar R}\times[-\bar h,\bar h]$.
Namely:
\beq
\label{172}
\max_{(y,h)\in\partial Q}I(Z_m[y,h])<c_{\Omega,m}\le \max_{(y,h)\in  Q}I(Z_m[y,h])<\min\left\{2,1+\frac2N\right\}c_m,
\eeq
where the last inequality has already been proved in Lemma \ref{L7}.

We first recall that $L_{\Omega_m}>c_m$, by Proposition \ref{P3.4}, and that, on the one hand, $L_{\Omega,m}$ does not depend on the choice of $\bar R$ in $Q$ and, on the other hand, $L_{\Omega,m}\le c_{\Omega,m}$.
Then, we observe that 
$$
\lim_{R\to\infty}\sup\{I(Z_m[y,h]) \ :\ h\in[-\bar{ h},\bar h], \ |y|=R\}=c_m,
$$
by Lemma \ref{L3.1}(iii) and \eqref{1638}. 
Hence,  \eqref{172} is proved, taking also into account that in \eqref{1453}  we can consider the trivial map $\gamma (y,h)={Z}_{m}  [y,h]$ for all $(y, h) \in  Q$.

By continuity, the estimates we found hold with $I_\eta$ in place of $I$, when $\eta\in  (\eta(m),1)$, for a suitable $\eta(m)\in(0,1)$.
Finally, we can work exactly as in Sections 2-5 to finish the proof.
 Observe that in this Theorem the Lagrange multiplier $\lambda$ cannot vanish, by assumption \eqref{NE}. 
\qed


\bigskip

\noindent {\bf Acknowledgements:}
The first author is supported by the EPSRC grant EP/W026597/1.
The second author is supported by the ``Gruppo Nazionale per l'Analisi Matematica, la Probabilit\`a e le loro Applicazioni (GNAMPA)'' of the {\em Istituto Nazionale di Alta Matematica (INdAM)}, by the MIUR Excellence Department Project MatMod@TOV  awarded to the Department of Mathematics, University of Rome Tor Vergata, CUP E83C23000330006, and by the  MUR-PRIN-2022AKNSE4\underline{ }003  ``Variational and Analytical aspects of Geometric PDEs''.


  

{\small

\begin{thebibliography}{10}

\bibitem{MR4232664}
L.~Appolloni and S.~Secchi.
\newblock Normalized solutions for the fractional {NLS} with mass supercritical
  nonlinearity.
\newblock {\em J. Differential Equations}, 286:248--283, 2021.

\bibitem{MR3539467}
T.~Bartsch, L.~Jeanjean, and N.~Soave.
\newblock Normalized solutions for a system of coupled cubic {S}chr\"odinger
  equations on {$\mathbb{R}^3$}.
\newblock {\em J. Math. Pures Appl. (9)}, 106(4):583--614, 2016.

\bibitem{MR4304693}
T.~Bartsch, R.~Molle, M.~Rizzi, and G.~Verzini.
\newblock Normalized solutions of mass supercritical {S}chr\"odinger equations
  with potential.
\newblock {\em Comm. Partial Differential Equations}, 46(9):1729--1756, 2021.

\bibitem{MR3895385}
T.~Bartsch and N.~Soave.
\newblock Multiple normalized solutions for a competing system of
  {S}chr\"odinger equations.
\newblock {\em Calc. Var. Partial Differential Equations}, 58(1):Paper No. 22,
  24, 2019.

\bibitem{BaWe05}
T.~Bartsch and T.~Weth.
\newblock Three nodal solutions of singularly perturbed elliptic equations on
  domains without topology.
\newblock {\em Ann. Inst. H. Poincaré Anal. Non Linéaire}, 22(3):259–281,
  2005.

\bibitem{BC87ARMA}
V.~Benci and G.~Cerami.
\newblock Positive solutions of some nonlinear elliptic problems in exterior
  domains.
\newblock {\em Arch. Rational Mech. Anal.}, 99(4):283--300, 1987.

\bibitem{BeLi83ARMA1}
H.~Berestycki and P.-L. Lions.
\newblock Nonlinear scalar field equations. {I}. {E}xistence of a ground state.
\newblock {\em Arch. Rational Mech. Anal.}, 82(4):313--345, 1983.

\bibitem{MR695536}
H.~Berestycki and P.-L. Lions.
\newblock Nonlinear scalar field equations. {II}. {E}xistence of infinitely
  many solutions.
\newblock {\em Arch. Rational Mech. Anal.}, 82(4):347--375, 1983.

\bibitem{MR4601303}
J.~Borthwick, X.~Chang, L.~Jeanjean, and N.~Soave.
\newblock Normalized solutions of {$L^2$}-supercritical {NLS} equations on
  noncompact metric graphs with localized nonlinearities.
\newblock {\em Nonlinearity}, 36(7):3776--3795, 2023.

\bibitem{MR4748624}
J.~Borthwick, X.~Chang, L.~Jeanjean, and N.~Soave.
\newblock Bounded {P}alais-{S}male sequences with {M}orse type information for
  some constrained functionals.
\newblock {\em Trans. Amer. Math. Soc.}, 377(6):4481--4517, 2024.

\bibitem{Ca03book}
T.~Cazenave.
\newblock {\em Semilinear {S}chr\"{o}dinger equations}, volume~10 of {\em
  Courant Lecture Notes in Mathematics}.
\newblock New York University, Courant Institute of Mathematical Sciences, New
  York; American Mathematical Society, Providence, RI, 2003.

\bibitem{CaLi82CMP}
T.~Cazenave and P.-L. Lions.
\newblock Orbital stability of standing waves for some nonlinear
  {S}chr\"{o}dinger equations.
\newblock {\em Comm. Math. Phys.}, 85(4):549--561, 1982.

\bibitem{CM10}
G.~Cerami and R.~Molle.
\newblock On some {S}chr\"{o}dinger equations with non regular potential at
  infinity.
\newblock {\em Discrete Contin. Dyn. Syst.}, 28(2):827--844, 2010.

\bibitem{MR1989833}
G.~Cerami and D.~Passaseo.
\newblock The effect of concentrating potentials in some singularly perturbed
  problems.
\newblock {\em Calc. Var. Partial Differential Equations}, 17(3):257--281,
  2003.

\bibitem{MR4755505}
X.~Chang, L.~Jeanjean, and N.~Soave.
\newblock Normalized solutions of {$L^2$}-supercritical {NLS} equations on
  compact metric graphs.
\newblock {\em Ann. Inst. H. Poincar\'e{} C Anal. Non Lin\'eaire},
  41(4):933--959, 2024.

\bibitem{cingolani2024}
S.~Cingolani, M.~Gallo, N.~Ikoma, and K.~Tanaka.
\newblock Normalized solutions for nonlinear {S}chr{\"o}dinger equations with
  ${L}^2$-critical nonlinearity.
\newblock {arXiv 2410.23733}, 2024.

\bibitem{MR1658565}
E.~N. Dancer.
\newblock Superlinear problems on domains with holes of asymptotic shape and
  exterior problems.
\newblock {\em Math. Z.}, 229(3):475--491, 1998.

\bibitem{MR3959930}
S.~Dovetta and L.~Tentarelli.
\newblock {$L^2$}-critical {NLS} on noncompact metric graphs with localized
  nonlinearity: topological and metric features.
\newblock {\em Calc. Var. Partial Differential Equations}, 58(3):Paper No. 108,
  26, 2019.

\bibitem{MR2825606}
P.~Esposito and M.~Petralla.
\newblock Pointwise blow-up phenomena for a {D}irichlet problem.
\newblock {\em Comm. Partial Differential Equations}, 36(9):1654--1682, 2011.

\bibitem{EvansBook}
L.~C. Evans.
\newblock {\em Partial differential equations}, volume~19 of {\em Graduate
  Studies in Mathematics}.
\newblock American Mathematical Society, Providence, RI, second edition, 2010.

\bibitem{GS81CPAM}
B.~Gidas and J.~Spruck.
\newblock Global and local behavior of positive solutions of nonlinear elliptic
  equations.
\newblock {\em Comm. Pure Appl. Math.}, 34(4):525--598, 1981.

\bibitem{MR619749}
B.~Gidas and J.~Spruck.
\newblock A priori bounds for positive solutions of nonlinear elliptic
  equations.
\newblock {\em Comm. Partial Differential Equations}, 6(8):883--901, 1981.

\bibitem{MR1814364}
D.~Gilbarg and N.~S. Trudinger.
\newblock {\em Elliptic partial differential equations of second order}.
\newblock Classics in Mathematics. Springer-Verlag, Berlin, 2001.
\newblock Reprint of the 1998 edition.

\bibitem{IkMi20CalcVar}
N.~Ikoma and Y.~Miyamoto.
\newblock Stable standing waves of nonlinear {S}chr\"{o}dinger equations with
  potentials and general nonlinearities.
\newblock {\em Calc. Var. Partial Differential Equations}, 59(2):Paper No. 48,
  20 pp,, 2020.

\bibitem{Je97NA}
L.~Jeanjean.
\newblock Existence of solutions with prescribed norm for semilinear elliptic
  equations.
\newblock {\em Nonlinear Anal.}, 28(10):1633--1659, 1997.

\bibitem{JJ97}
L.~Jeanjean.
\newblock Existence of solutions with prescribed norm for semilinear elliptic
  equations.
\newblock {\em Nonlinear Anal.}, 28(10):1633–1659, 1997.

\bibitem{MR4150876}
L.~Jeanjean and S.-S. Lu.
\newblock A mass supercritical problem revisited.
\newblock {\em Calc. Var. Partial Differential Equations}, 59(5):Paper No. 174,
  43, 2020.

\bibitem{MR4701352}
L.~Jeanjean, J.~Zhang, and X.~Zhong.
\newblock A global branch approach to normalized solutions for the
  {S}chr\"odinger equation.
\newblock {\em J. Math. Pures Appl. (9)}, 183:44--75, 2024.

\bibitem{MR0969899}
M.~K. Kwong.
\newblock Uniqueness of positive solutions of {$\Delta u-u+u^p=0$} in {${\bf
  R}^n$}.
\newblock {\em Arch. Rational Mech. Anal.}, 105(3):243--266, 1989.

\bibitem{LM_PRSE_23}
S.~Lancelotti and R.~Molle.
\newblock Normalized positive solutions for schrödinger equations with
  potentials in unbounded domains.
\newblock {\em Proceedings of the Royal Society of Edinburgh: Section A
  Mathematics}, 154(5):1518–1551, 2024.

\bibitem{MR0778974}
P.-L. Lions.
\newblock The concentration-compactness principle in the calculus of
  variations. {T}he locally compact case. {II}.
\newblock {\em Ann. Inst. H. Poincar\'{e} Anal. Non Lin\'{e}aire},
  1(4):223--283, 1984.

\bibitem{MRV22JDE}
R.~Molle, G.~Riey, and G.~Verzini.
\newblock Normalized solutions to mass supercritical {S}chr\"{o}dinger
  equations with negative potential.
\newblock {\em J. Differential Equations}, 333:302--331, 2022.

\bibitem{MR3547674}
E.~Serra and L.~Tentarelli.
\newblock On the lack of bound states for certain {NLS} equations on metric
  graphs.
\newblock {\em Nonlinear Anal.}, 145:68--82, 2016.

\bibitem{MR4107073}
N.~Soave.
\newblock Normalized ground states for the {NLS} equation with combined
  nonlinearities.
\newblock {\em J. Differential Equations}, 269(9):6941--6987, 2020.

\bibitem{MR4096725}
N.~Soave.
\newblock Normalized ground states for the {NLS} equation with combined
  nonlinearities: the {S}obolev critical case.
\newblock {\em J. Funct. Anal.}, 279(6):108610, 43, 2020.

\bibitem{ZhZh22NODEA}
Z.~Zhang and Z.~Zhang.
\newblock Normalized solutions of mass subcritical {S}chr\"{o}dinger equations
  in exterior domains.
\newblock {\em NoDEA Nonlinear Differential Equations Appl.}, 29(3):Paper No.
  32, 25, 2022.

\end{thebibliography}

}
\end{document}